\documentclass[a4paper,12pt]{amsart}
\usepackage{amsmath,amssymb}

\numberwithin{equation}{section}

\newtheorem{thm}{Theorem}[section]
\newtheorem{prop}[thm]{Proposition}
\newtheorem{lemma}[thm]{Lemma}
\newtheorem{cor}[thm]{Corollary}

\theoremstyle{definition}

\theoremstyle{remark}
\newtheorem{rmk}[thm]{Remark}
\newtheorem{ex}[thm]{Example}

\newtheorem*{con}{Conventions}

\newcommand{\C}{\mathbb{C}}
\newcommand{\N}{\mathbb{N}}
\newcommand{\R}{\mathbb{R}}
\newcommand{\T}{\mathbb{T}}
\newcommand{\Z}{\mathbb{Z}}

\newcommand{\FF}{\mathcal{F}}
\newcommand{\KK}{\mathcal{K}}

\newcommand{\TT}{\mathcal{T}}
\newcommand{\OO}{\mathcal{O}}

\newcommand{\im}{{\operatorname{Im}}}

\newcommand{\Ind}{\operatorname{Ind}}

\newcommand{\lsp}{\operatorname{span}}
\newcommand{\clsp}{\overline{\operatorname{span}}}

\newcommand{{\Kum}}{\operatorname{Kum}}

\newcommand{\stem}{\operatorname{stem}}
\newcommand{\Aut}{\operatorname{Aut}}
\newcommand{\BS}{\operatorname{BS}}

\newcommand{\rt}{\textrm{r}}
\newcommand{\lt}{\textrm{l}}
\newcommand{\op}{\textrm{op}}

\begin{document}

\title[Toeplitz algebras
of Baumslag-Solitar semigroups]{Phase transitions on the Toeplitz algebras\\
of Baumslag-Solitar semigroups}
\author[Clark]{Lisa Orloff Clark}
\author[an Huef]{Astrid an Huef}
\author[Raeburn]{Iain Raeburn}
\address{Department of Mathematics and Statistics, University of Otago, PO Box 56, Dunedin 9054, New Zealand.}
\email{\{lclark, astrid, iraeburn\}@maths.otago.ac.nz} 
\thanks{This research was supported by the Marsden Fund of the Royal Society of New Zealand.}
\date{11 March 2015}

\begin{abstract}
Spielberg has recently shown that Baumslag-Solitar groups associated to pairs of positive integers are quasi-lattice ordered 
in the sense of  Nica. Thus they have  tractable Toeplitz algebras. Each of these algebras carries a natural dynamics. 
Here we construct the equilibrium states (the KMS states) for these dynamics. For inverse temperatures larger than a critical value, 
there is a large simplex of KMS states parametrised by probability measures on the unit circle. At the critical value, and under a 
mild hypothesis, there is a phase transition in which this simplex collapses to a singleton. There is a further phase transition 
at infinity, in the sense that there are many ground states which cannot be realised as limits of KMS states with finite inverse temperatures.
 \end{abstract}
 
 \maketitle
 
\section{Introduction}

Spielberg \cite{SCP} has recently studied a large family of $C^*$-algebras which includes the $C^*$-algebras of higher-rank graphs \cite{KP, RSY2} and the boundary quotients of quasi-lattice ordered groups \cite{N, CL2}. He has also shown that the Baumslag-Solitar groups are quasi-lattice ordered with boundary quotients that are typically Kirchberg algebras, and has computed the $K$-theory of these boundary quotients \cite{SBS}.

A quasi-lattice ordered group also has a (much larger) Toeplitz algebra, and the Toeplitz algebras of groups similar to Baumslag-Solitar groups have recently  been shown to exhibit interesting phase transitions. Indeed, there is nontrivial overlap\footnote{More precisely, the groups $\BS(1,d)$ are discussed in \cite[\S9]{LRR}. There is similar overlap between the structural results and $K$-theory computations in \cite{EaHR} and \cite{SBS}.} between the Toeplitz algebras studied in \cite{LR, LRR} and the Toeplitz algebras of the Baumslag-Solitar groups studied in \cite{SBS} (see \cite[\S9]{LRR}). So one naturally wonders whether there are interesting phase transitions on the Toeplitz algebras of Baumslag-Solitar groups. Here we confirm that this is indeed the case.

Suppose that $c$ and $d$ are nonzero integers. The Baumslag-Solitar group $G=\BS(c,d)$ is the group generated by two elements $a,b$ subject only to the relation $ab^c=b^da$. When $c$ and $d$ are positive, we consider  the subsemigroup $P$ of $G$ generated by $a$ and $b$. This semigroup defines a partial order on $G$: $g\leq h$ means that $g^{-1}h\in P$. Spielberg proved in \cite[Theorem~2.11]{SBS} that the pair $(G,P)$ is quasi-lattice ordered in the sense of Nica \cite{N}. The Toeplitz algebra is the $C^*$-subalgebra $\TT(P)$ of $B(\ell^2(P))$ generated by a left-regular representation of $P$ by isometries $\{T_x:x\in P\}$ (but see \S\ref{back} for further discussion of our conventions). This algebra carries a natural gauge action of the circle, which we can lift to an action $\alpha$ of $\R$. We are interested in the KMS states of the dynamical system $(\TT(P),\alpha)$.

We show that for inverse temperatures $\beta$ larger than $\ln d$, there is a large simplex of KMS$_\beta$ states parametrised by the probability measures on the circle.  When $d$ does not divide $c$, there is a phase transition at the \emph{critical inverse temperature} $\ln d$ in which this simplex collapses to a single point, and the KMS$_{\ln d}$ state factors through the boundary quotient of \cite{CL2}. The condition ``$d$ does not divide $c$'' has previously occurred in Spielberg's analysis of the groupoid model for the boundary quotient, where it is shown to be necessary and sufficient for the groupoid to be topologically principal (which he calls ``essentially free'') \cite[Theorem~4.9]{SBS}. 

We begin with a section on background material: we discuss our conventions concerning quasi-lattice ordered groups and their Toeplitz algebras, and the normal form for elements of Baumslag-Solitar groups which we will use throughout. The normal form identifies a family of words in $P$ that play a vital role in computations in $G$ and $P$. We call these words ``stems'', and in \S\ref{sect:stems} we establish  some properties of the map which sends an arbitary element of $P$ to its stem. In \S\ref{sect:pres} we give a presentation of our Toeplitz algebra which will allow us to build Hilbert-space representations. Then in \S\ref{sect:char}, we turn to KMS states. The Toeplitz algebra $\TT(P)=C^*(\{T_x:x\in P\})$ is spanned by the elements $T_xT_y^*$, and the KMS states are the states that satisfy a commutation relation involving products of two spanning elements. In Proposition~\ref{charKMS} we give a characterisation of KMS states in terms of their values on individual spanning elements. This implies, for example, 
that all KMS states at real inverse temperatures factor through the quotient in which the generator $T_b$ is unitary. 

Our main theorem about the KMS$_\beta$ states for $\beta>\ln d$ is Theorem~\ref{KMSToe}, and the rest of \S\ref{sect:large} is devoted to its proof. The strategy is a refinement of the one developed in \cite{LR} and  \cite{LRR}. To build KMS states, we exploit that all KMS states think $T_b$ is unitary: we take a carefully chosen unitary representation $W$ of the subgroup generated by $b$, and induce it to a large unitary representation $\Ind W$ of $G$. The KMS states come from the isometric representation obtained by restricting $(\Ind W)|_P$ to a suitable invariant (but not reducing!) subspace. Our results at the critical inverse temperature are in Proposition~\ref{KMScrit}, and we show by example that they are sharp: when $d$ divides $c$, there is more than one KMS$_{\ln d}$ state.  

Our last main result is Theorem~\ref{thmground}, where we identify the ground and KMS$_\infty$ states of our system. This seems to be harder than in previous computations of KMS structure: ground states need not factor through the same quotient of $\TT(P)$, and hence we cannot use induced representations. But by mimicking what happens in \S\ref{sect:large}, we can build suitable isometric representations with our bare hands. We close with an appendix in which we prove that the quasi-lattice ordered group $(G,P)$ is amenable in the sense of \cite{N,LRold}. This result is not strictly needed in the rest of the paper, but it does simplify things notationally because it implies that the Toeplitz algebra is universal for Nica-covariant representations of $P$ (see Corollary~\ref{same}).

\section{Background}\label{back}

\subsection{Quasi-lattice ordered groups}
Suppose that $G$ is a group and $P$ is a subsemigroup such that $P\cap P^{-1}=\{e\}$.
Then there is a partial order on $G$ such that 
\[
g\leq h\Longleftrightarrow h\in gP\Longleftrightarrow g^{-1}h\in P.
\]
This partial order is left-invariant, in the sense that $g\leq h\Longrightarrow kg\leq kh$. 

According to Nica \cite{N}, the pair $(G,P)$ is a \emph{quasi-lattice ordered group} if every pair $g,h$ in $G$ with a common upper bound in $P$ has a least upper bound $g\vee h$ in $P$. Subsequently, Crisp and Laca showed that it suffices to check that every element $g\in G$ with an upper bound in $P$ has a least upper bound in $P$ \cite[Lemma~7]{CL1} (and that useful lemma contains several other equivalent reformulations of the definition). We write $g\vee h<\infty$ if $g$ and $h$ have an upper bound in $P$, and $g\vee h=\infty$ otherwise.

Suppose that $(G,P)$ is quasi-lattice ordered. We consider the Hilbert space $\ell^2(P)$ with the orthonormal basis $\{e_x:x\in P\}$ of point masses. For each $x\in P$, there is an isometry $T_x$ on $\ell^2(P)$ such that $T_xe_y=e_{xy}$ for $y\in P$. We have $T_e=1$ (the identity operator), and $T_xT_y=T_{xy}$. In other words, $T$ is a homomorphism of the monoid $P$ into the monoid of isometries on $\ell^2(P)$, and we say that $T$ is an \emph{isometric representation} of $P$. Nica observed that the representation $T$ has the extra property 
\begin{equation}\label{Nicacov}
T_xT_x^*T_yT_y^*=\begin{cases}T_{x \vee y} T_{x \vee y}^*&\text{if $x\vee y<\infty$}\\
0&\text{if $x\vee y=\infty$.}
\end{cases}
\end{equation}
Now we say that an isometric representation satisfying \eqref{Nicacov} is  \emph{Nica covariant}. Nica covariance is equivalent to
\begin{equation}\label{altNica}
T_x^*T_y =\begin{cases}
T_{x^{-1}(x \vee y)}T_{y^{-1}(x \vee y)}^*&\text{if $x\vee y<\infty$}\\
0&\text{if $x\vee y=\infty$.}
\end{cases}
\end{equation}

A quasi-lattice ordered group $(G,P)$ has two $C^*$-algebras: the \emph{Toeplitz algebra} $\TT(P)$ is the $C^*$-subalgebra of $B(\ell^2(P))$ generated by the operators $\{T_x:x\in P\}$, and the \emph{universal $C^*$-algebra} $C^*(G,P)$ is generated by a universal Nica-covariant representation $i:P\to C^*(G,P)$. Nica covariance implies that every word in the $i(x)$ and their adjoints reduces to one of the form $i(x)i(y)^*$, and hence
\[
C^*(G,P)=\clsp\{i(x)i(y)^*:x,y\in P\}.
\]
The Toeplitz representation $T:P\to \TT(P)$ induces a surjection $\pi_T:C^*(G,P)\to \TT(P)$, and a major issue considered in \cite[\S4]{N} is when $\pi_T$ is an isomorphism. 

Because $x\vee y=y\vee x$, Nica covariance implies that the range projections $i(x)i(x)^*$ commute with each other, and then 
$D:=\clsp \{i(x)i(x)^*:x\in P\}$ is a commutative $C^*$-subalgebra. There is a positive norm-decreasing linear map $E:C^*(G,P)\to D$ such that $E(i(x)i(y)^*)=\delta_{x,y}i(x)i(x)^*$ (see \cite[\S4.2]{N} or \cite[Proposition~3.1]{LRold}), and we say that $(G,P)$ is \emph{amenable} if $E$ is faithful. Nica proved that if $(G,P)$ is amenable, then the Toeplitz representation $\pi_T$ is injective (see  \cite[\S4.2]{N} or \cite[Corollary~3.9]{LRold}). This implies that the Toeplitz algebra has the universal property of $(C^*(G,P),i)$, and justifies the following:

\begin{con}
All the quasi-lattice ordered groups $(G,P)$ in this paper are 
amenable (see Theorem~\ref{BSamenable}). 
So it makes no difference whether we use  $C^*(G,P)$ or $\TT(P)$. We choose to write $C^*(G,P)$ for 
the algebra because we want to emphasise the universal property, but write $T$ for the universal Nica-covariant representation of $P$ in $C^*(G,P)$. We write $\N=\{0,1,2,\dots\}$. 
\end{con}

\subsection{Baumslag-Solitar groups} We fix positive integers $c$ and $d$. Then the \emph{Baumslag-Solitar group} is the group
\[
G := \langle a, b : ab^c = b^da \rangle;
\] 
if we want to emphasise the dependence on the numbers $c,d$, we write $G=\BS(c,d)$. We consider the submonoid $P$ of $G$ generated  by
by $a$ and $b$. Spielberg proved in \cite[Theorem~2.11]{SBS} that (provided $c$ and $d$ are positive) $(G,P)$ is quasi-lattice ordered. For the rest of this paper, $(G,P)$ denotes one of these groups.

Following \cite{SBS}, we write $\theta$ for the homomorphism $\theta:G \to \Z$ 
such that $\theta(a)=1$ and $\theta(b) = 0$, and call $\theta(g)$ the \emph{height} of $g$. Baumslag-Solitar 
groups are examples of Higman-Neumann-Neumann extensions, and each element has a unique normal form 
\[
g=b^{s_0}a^{\varepsilon_1}b^{s_1}\cdots b^{s_{k-1}}a^{\varepsilon_k}b^{s_k}
\]
in which each $\varepsilon_i$ is $\pm 1$, $0\leq s_{i-1}<d$ when $\varepsilon_i=1$, and $0\leq s_{i-1}<c$ when $\varepsilon_i=-1$ (see, for example, Theorem~2.1 on \cite[page~182]{LS}).  For elements of $P$, we have $\varepsilon_i =1$ for all $k$ and $0\leq s_i<d$ for all $i<k$; there is no restriction on $b^{s_k}$ except that $s_k\geq 0$, and we have $k=\theta(x)$. 

\section{Stems and their properties}\label{sect:stems}

Each $x\in P$ has a unique normal form $x=b^{s_0}a b^{s_1}a\cdots b^{s_{k-1}}ab^{s_k}$ with $0\leq s_i<d$ for $i<k$ and $k=\theta(x)$. 
We then write 
\[
\stem(x):=b^{s_0}a b^{s_1}a\cdots b^{s_{k-1}}a
\]
for the \emph{stem} of $x$. We write $\Sigma_k$ for the set of possible stems with height $k$, including $\Sigma_0:=\{e\}$; note that each $\Sigma_k$ is finite with cardinality $d^k$.

Our constructions of KMS states will involve the Hilbert space $\bigoplus_{k\geq 0}\ell^2(\Sigma_k)$, and hence properties of stems will be important throughout the paper. In this section, we describe some of these properties.

\begin{lemma}\label{forlisa}
Suppose that $x,y\in P$. Then $\stem(x\stem(y))=\stem(xy)$. If $s$ and $t$ satisfy $y=\stem(y)b^t$ and $x\stem(y)=\stem(x\stem(y))b^s$, then $xy=\stem(xy)b^{s+t}$.
\end{lemma}

\begin{proof}
There exists $t$ such that $y=\stem(y)b^t$, and then 
\[
xy=x\stem(y)b^t=(\stem(x\stem(y))b^s)b^t\quad\text{for some $s\in \N$.}
\]
On the other hand, we have $xy=\stem(xy)b^r$ for some $r\in \N$. Since both $\stem(xy)b^r$ and $\stem(x\stem(y))b^{s+t}$ are normal forms for $xy$, the uniqueness of normal forms implies that $\stem(x\stem(y))=\stem(xy)$ and that $r=s+t$. 
\end{proof}

\begin{lemma}\label{maponstems}
\begin{enumerate}
\item\label{hkm} For all $k\in \N$ and $m\geq 0$, the map $h_{k,m}:x\mapsto \stem(b^mx)$ is a bijection of $\Sigma_k$ onto $\Sigma_k$.

\item\label{gk} For all $k\in \N$, the map $x\mapsto \stem(b^cax)$ is an injection of $\Sigma_k$ into $\Sigma_{k+1}$.
\end{enumerate}
\end{lemma}

\begin{proof}[Proof of part~\eqref{hkm}]
 We prove by induction on $k$ that $h_{k,m}$ is a bijection for all $m\in \N$. The result is trivial if $k=0$. 
For $k=1$, $h_{1,m}(b^ia)=b^ja$ where $m+i=nd+j$ and $0\leq j<d$; since addition by $m$ modulo $d$ is a bijection on $\{0,1,\cdots, d-1\}$, $h_{1,m}$ is a bijection. Suppose that $h_{k,m}$ is a bijection for all $m$. Because $\Sigma_{k+1}$ is finite, 
it suffices to see that $h_{k+1,m}$ is one-to-one. So suppose that $h_{k+1,m}(x)=h_{k+1,m}(x')$. We can write 
$x=b^iay$ for some $y\in \Sigma_k$.  We now define $n,j$ by  $m+i=nd+j$ and $0\leq j<d$, and then 
\[
h_{k+1,m}(b^iay)=\stem(b^{m+i}ay)=b^ja\stem(b^{nc}y)=b^ja h_{k,nc}(y).
\]
Doing the same for $x'=b^{i'}ay'$ shows that
\[
b^ja h_{k,nc}(y)=b^{j'}ah_{k,n'c}(y')\text{ where $m+i=nd+j$ and $m+i'=n'd+j'$}.
\]
Now uniqueness of the normal form forces $j=j'$ and $h_{k,nc}(y)=h_{k,n'c}(y')$. Since $j=j'$ and 
$|i-i'|<d$, we must have $i=i'$ and $n=n'$ too. Now the injectivity of $h_{k,nc}$ implies that $y=y'$ and $x=b^jay=x'$.
\end{proof}

For the proof of part~\eqref{gk}, we separate out a calculation. Notice that it applies with $m=c$, 
and then gives \eqref{gk} for $k=1$. (Since $\Sigma_0=\{e\}$, \eqref{gk} is trivially true for $k=0$.)

\begin{lemma}\label{sublem}
Suppose that $m\in \N$, $b^ia, b^ja\in \Sigma_1$, and $\stem(b^mab^ia)=\stem (b^mab^ja)$. Then $i=j$.
\end{lemma}

\begin{proof}
Write $m=nd+m'$ with $0\leq m'<d$. Then 
\[
b^mab^ia=b^{m'}ab^{nc+i}a\ \text{ and }\ b^mab^ja=b^{m'}ab^{nc+j}a.
\]
Now we write $nc +i=n_id+i'$ and $nc+j=n_jd+j'$, and the hypothesis gives
\[
b^{m'}ab^{i'}a=\stem(b^mab^ia)=\stem (b^mab^ja)=b^{m'}ab^{j'}a.
\]
Thus the uniqueness of the normal form implies that $i'=j'$, and we have
\[
i=n_id+i'-nc=n_jd+j'-nc +(n_i-n_j)d=j+(n_i-n_j)d.
\]
Now $0\leq i,j<d$ implies that $n_i-n_j=0$, and hence $i=j$.
\end{proof}

\begin{proof}[Proof of Lemma~\ref{maponstems}\,\eqref{gk}]
Suppose that $x,y\in \Sigma_k$ and $\stem(b^cax)=\stem(b^cay)$. We write $x$ and $y$ in normal form as
\[
x=b^{x_0}ab^{x_1}a\cdots b^{x_{k-1}}a\ \text{ and }\ y=b^{y_0}ab^{y_1}a\cdots b^{y_{k-1}}a,
\]
and prove by induction on $n$ that $x_i=y_i$ for $0\leq i\leq n<k$. Since the stem of $b^cax$ 
begins with the stem of $b^cab^{x_0}a$, and similarly for $b^cay$, Lemma~\ref{sublem} 
implies that $x_0=y_0$. Suppose that we have $x_i=y_i$ for $i\leq n<k-1$. Next we put into normal form
\[
b^cab^{x_0}ab^{x_1}a\cdots b^{x_n}a=b^{s_0}ab^{s_1}a\cdots b^{s_n}ab^m;
\]
by the inductive hypothesis, we have
\begin{align*}
b^cax&=b^{s_0}ab^{s_1}a\cdots b^{s_n}ab^mb^{x_{n+1}}\cdots b^{x_{k-1}}a,\text{ and }\\ 
b^cay&=b^{s_0}ab^{s_1}a\cdots b^{s_n}ab^mb^{y_{n+1}}\cdots b^{y_{k-1}}a.
\end{align*}
Thus
\begin{align*}
b^{s_0}ab^{s_1}&a\cdots b^{s_n}a\stem(b^mab^{x_{n+1}}a\cdots b^{x_{k-1}}a)=\stem(b^cax)\\
&=\stem(b^cay)=b^{s_0}ab^{s_1}a\cdots b^{s_n}a\stem(b^mab^{y_{n+1}}a\cdots b^{y_{k-1}}a),
\end{align*}
and we deduce that 
\[
\stem(b^mab^{x_{n+1}}a\cdots b^{x_{k-1}}a)=\stem(b^mab^{y_{n+1}}a\cdots b^{y_{k-1}}a).
\]
In particular, we have $\stem(b^mab^{x_{n+1}}a)=\stem(b^mab^{y_{n+1}}a)$, 
and Lemma~\ref{sublem} implies that $x_{n+1}=y_{n+1}$.
\end{proof}

\begin{lemma}
 \label{lem:iains}Let $x,y \in P$ such that $x \vee y < \infty$.  
\begin{enumerate} 
 \item\label{it:iains1} If $\theta(y) > \theta(x)$ then there exists $t \in \N$ such that 
$x \vee y =yb^t$.
\item\label{it:iains2} If $\theta(x) = \theta(y)$ then there exists $t \in \N$ such that either 
\[x \vee y = x = yb^t\text{ or } x \vee y =y=xb^t.\] 
\end{enumerate}
\end{lemma}

\begin{proof}
For (\ref{it:iains1}), suppose $x = \stem(x)b^s$ for some $s \in \N$.  Then because $x\vee y < \infty$ and $\theta(y) > \theta(x)$,  
$\stem(y)=\stem(x)\sigma$ for some stem $\sigma$. Then 
$y = \stem(x)\sigma b^n$ for some $n \in \N$. 
Now choose a stem $\tau$ such that $\stem(b^s\tau)=\sigma$ by Lemma~\ref{maponstems}(\ref{hkm}) (using that the map $h_{\theta(\sigma),s}$ is surjective  and so $\sigma$ must be in the image).  
That is, $b^s\tau = \sigma b^r$ for some $r \in \N$.
Then
\[x\tau=\stem(x)b^s\tau=\stem(x)\sigma b^r.\]
Therefore
\[x\vee y = \stem(x)\sigma b^{\max(n,r)}=yb^{\max(n,r)-n}\]
so if we let $t=\max(n,r)-n$, then $x\vee y = yb^t$.

For part (\ref{it:iains2}), if $x\vee y < \infty$ and $\theta(y) = \theta(x)$, then putting $x \vee y$ into normal form tells us
that $\stem(x) = \stem(y)$.  The result follows.  
\end{proof}

\section{A presentation for the Toeplitz algebra}\label{sect:pres}

We want to build representations of $C^*(G,P)$. For this we use:

\begin{prop}\label{defrel}
Suppose that $\pi:C^*(G,P)\to B$ is a homomorphism. Then $U:=\pi(T_b)$ and $V:=\pi(T_a)$ are isometries, and satisfy
\begin{enumerate}
\item\label{t1} $VU^c=U^dV$;
\item\label{t4} $U^*V=U^{d-1}VU^{*c}$;
\item\label{t5} $V^*U^jV=0$ for $1\leq j<d$.
\end{enumerate}
Conversely, if $U$ and $V$ are isometries in a $C^*$-algebra $B$ satisfying \eqref{t1}, \eqref{t4} and \eqref{t5}, then there is a Nica covariant representation $S:P\to B$ such that $S_a=V$ and $S_b=U$, and a homomorphism $\pi_{U,V}:C^*(G,P)\to B$ such that $U=\pi_{U,V}(T_b)$ and $V=\pi_{U,V}(T_a)$. 
\end{prop}

Let $\pi:C^*(G,P)\to B$ be a homomorphism. Then $\pi\circ T$ is a Nica covariant representation of $(G,P)$.
The relation \eqref{t1} follows because $ab^c=b^da$ in $P$. The relation \eqref{t4} follows from Nica covariance for the pair $(b,a)$, which has $b\vee a=ab^c$, so that
\[
U^*V=\pi(T_b^*T_a)=\pi(T_{b^{-1}(b\vee a)}T_{a^{-1}(b\vee a)}^*)=\pi(T_{b^{d-1}a})\pi(T_{b^c})^*=U^{d-1}VU^{*c}.
\]
The relation \eqref{t5} is Nica covariance for $(a,b^ja)$, for which we have $a\vee b^ja=\infty$. So it remains for us to prove the converse.

\begin{rmk}\label{r1} The relation \eqref{t5} is equivalent to saying that $\{U^jV:0\leq j<d\}$ is a Toeplitz-Cuntz family: in other words, the $U^jV$ are isometries satisfying
\[
1\geq \sum_{j=0}^{d-1} (U^jV)(U^jV)^*.
\]
For $k\geq 1$, the stems of height $k$ are precisely the words of length $k$ in the alphabet $\{b^ja:0\leq j<d\}$, and for $\sigma=b^{j_0}ab^{j_1}a \dots b^{j_{k-1}}a$ we have
\[
\pi(T_\sigma)=(U^{j_0}V)(U^{j_1}V) \dots (U^{j_{k-1}}V).
\]
Thus $\{\pi(T_\sigma):\sigma\in \Sigma_k\}$ is also a Toeplitz-Cuntz family for each $k\geq 1$.
\end{rmk}

\begin{rmk}\label{r2}
Suppose that $U,V$ satisfy relations \eqref{t1} and \eqref{t5} of Proposition~\ref{defrel} and $U$ is unitary. Then multiplying \eqref{t1} on the left by $U^*$ and the right by $U^{*c}$ gives \eqref{t4}. (This argument  uses the extra relation $UU^*=1$, so it does not work when $U$ is just an isometry.) Our relation \eqref{t1} is relation (3) in \cite[Theorem~3.23]{SBS}. In \cite[Remark~3.24]{SBS}, Spielberg suggests that (3) and the Toeplitz-Cuntz relation equivalent to \eqref{t5} give a presentation of his Toeplitz algebra $\TT(G,P)$. However, we think that his $\TT(G,P)$ is intended to be $C^*(G,P)$, and that the extra relation \eqref{t4} is required for that. 
\end{rmk}

The hard bit in Proposition~\ref{defrel} is proving that a pair $(U,V)$ of isometries satisfying the relations gives us a Nica-covariant isometric representation $S$ of $P$ in $B$. It is clear how to define $S$: write $x\in P$ in  normal form $b^{s_0}ab^{s_1}a\cdots ab^{s_m}$, and define
\[
S_x:=U^{s_0}VU^{s_1}V\cdots V U^{s_m};
\]
we also set $S_e:=1$.
For $x,y\in P$, the product of normal forms is not necessarily a normal form, so to see that $S$ is multiplicative, we need to put the product
\[
xy=(b^{s_0}ab^{s_1}a\cdots ab^{s_m})(b^{t_0}ab^{t_1}a\cdots ab^{t_n})
\]
in normal form. However, this entails pulling any factors of the form $b^{kd}$ in $b^{s_m}b^{t_0}$ to the right, using the relation $b^da=ab^c$ to pull any such factors across each $a$ in turn. We can perform exactly the same calculations in 
\[
S_xS_y=U^{s_0}VU^{s_1}V\cdots V U^{s_m}U^{t_0}VU^{t_1}V\cdots V U^{t_n}
\]
using the relation \eqref{t1}, arriving at the formula for $S_{xy}$. So $S$ is multiplicative. 

To see that $S$ is Nica covariant, we begin with a special case. To avoid losing detail in subscripts, we say that a pair $(x,y)$ in $P$ is \emph{Nica covariant} when $S_x$, $S_y$ satisfy the Nica covariance relation \eqref{altNica}.

\begin{lemma}\label{height1Nica}
For every $s,t\in \N$, the pair $(b^s,ab^t)$ is Nica covariant.
\end{lemma}

\begin{proof}
Write $x=b^s$ and $y=ab^t$, and write $s=(n-1)d+j$ with $1\leq j\leq d$. Then we have
\[
x\vee y=
\begin{cases}ab^{nc}=xb^{d-j}a=yb^{nc-t}&\text{if $nc>t$}\\
y=xb^{d-j}ab^{t-nc}&\text{if $nc\leq t$,}
\end{cases}
\]
and 
\[
S_{x^{-1}(x \vee y)}S_{y^{-1}(x \vee y)}^*
=\begin{cases}U^{d-j}VU^{*(nc-t)}&\text{if $nc>t$}\\
U^{d-j}VU^{t-nc}&\text{if $nc\leq t$.}
\end{cases}
\]
Next we observe that \eqref{t4} implies
\begin{equation}\label{t4'}
U^{*r}U^{d-1}VU^{*c}=U^{*(r+1)}V\quad\text{for every integer $r\geq 0$.}
\end{equation}
Using this, we compute
\begin{align*}
S_x^*S_y&=U^{*((n-1)d+j)}VU^t\\
&=U^{*j}U^{*(n-1)d}VU^t\\
&=U^{*j}VU^{*(n-1)c}U^t\quad\text{by \eqref{t4'} with $r=d-1$, $n-1$ times}\\
&=(U^{d-j}VU^{*c})U^{*(n-1)c}U^t\quad\text{by \eqref{t4'} with $r=j-1$,}
\end{align*}
which is $U^{d-j}VU^{*(nc-t)}$ if $nc>t$ and $U^{d-j}VU^{t-nc}$ if $nc\leq t$.
\end{proof}

The next lemma will allow us to bootstrap Lemma~\ref{height1Nica} to longer words.

\begin{lemma}\label{height2Nica}
Suppose that $(x,y)$ is a Nica-covariant pair with $x\vee y<\infty$ and $\theta(x)\leq \theta(y)$. If $w$ has the form $ab^t$, then $(x,yw)$ is a Nica-covariant pair.
\end{lemma}

\begin{proof}
We have
\[
S_x^*S_{yw}=(S_x^*S_{y})S_w=(S_{x^{-1}(x \vee y)}S_{y^{-1}(x \vee y)}^*)S_{w}.
\]
The assumption $\theta(x)\leq \theta(y)$ implies that $x\vee y=$ has the form $yb^s$ (see Lemma~\ref{lem:iains}), and hence Lemma~\ref{height1Nica} implies that $(y^{-1}(x\vee y),w)=(b^s,w)$ is Nica covariant. Thus
\begin{align*}
S_x^*S_{yw}&=S_{x^{-1}(x \vee y)}(S_{y^{-1}(x \vee y)}^*S_{w})\\
&=S_{x^{-1}(x \vee y)}\big(S_{(y^{-1}(x\vee y))^{-1}(y^{-1}(x\vee y)\vee w)}S^*_{w^{-1}(y^{-1}(x\vee y)\vee w)}\big)\\
&=S_{x^{-1}y(y^{-1}(x\vee y)\vee w)}S^*_{w^{-1}(y^{-1}(x\vee y)\vee w)}.
\end{align*}
Now we recall that the partial order on $(G,P)$ is left invariant, and hence
\begin{align*}
&y(y^{-1}(x\vee y)\vee w)=(yy^{-1}(x\vee y))\vee yw=(x\vee y)\vee yw=x\vee yw\\
\intertext{and}
&y^{-1}(x\vee y)\vee w=y^{-1}(x\vee y)\vee y^{-1}yw=y^{-1}((x\vee y)\vee yw)=y^{-1}(x\vee yw).
\end{align*}
Thus
\[
S_x^*S_{yw}=S_{x^{-1}(x\vee yw)}S^*_{w^{-1}y^{-1}(x\vee yw)}=S_{x^{-1}(x\vee yw)}S^*_{(yw)^{-1}(x\vee yw)},
\]
as required.
\end{proof}

\begin{proof}[Proof of Proposition~\ref{defrel}] It remains for us to prove  that the representation $S$ is Nica covariant. Suppose that $x,y\in P$. It suffices to prove \eqref{altNica} when $\theta(x)\leq \theta(y)$ (otherwise take adjoints). First we suppose that $x\vee y=\infty$. We claim that $\stem(x)$ is not an initial segment of $\stem(y)$. To see this, suppose to the contrary that $\stem(y)=\stem(x)p$ and $x=\stem(x)b^t$. Then Lemma~\ref{maponstems}\eqref{hkm} implies that there is a stem $q$ such that $b^tq$ has the form $pb^s$. But them $xq=\stem(x)pb^s$ and $y$ has the same stem, and we can find a common upper bound for $x$ and $y$ of the form $\stem(x)pb^r$. Thus we have a contradiction, and the claim is proved. So there are distinct stems $\sigma, \tau$ in $\Sigma_{\theta(x)}$ such that $x$ has the form $x=\sigma b^s$ and $y=\tau p$. Then because $\{S_\rho=\pi(T_\sigma):\rho\in\Sigma_{\theta(x)}\}$ 
is a Toeplitz-Cuntz family (Remark~\ref{r1}), we have
\[
S_x^*S_y=S_{b^s}^*S_{\sigma}^*S_\tau S_p=0,
\]
as required in \eqref{altNica}. 

Next we suppose that $x\vee y<\infty$, in which case we have $x=\sigma b^s$ for $\sigma=\stem(x)$, and $y$ has the form $\sigma w$ for some $w\in P$ by uniqueness of the normal form. Then
\[
S_x^*S_y=S_{b^s}^*S_{\sigma}^*S_\sigma S_w=S_{b^s}^*S_w;
\]
since left invariance of the partial order gives \[x^{-1}(x\vee y)=b^{-s}(b^s\vee w)\text{\ and\ }y^{-1}(x\vee y)=w^{-1}(b^s\vee w),\] it suffices to prove the result for $x=b^s$. Now we trivially have Nica covariance for $(b^s,b^r)$, and Lemma~\ref{height2Nica} gives Nica covariance for $(b^s,b^rab^t)$. Now an induction argument using Lemma~\ref{height2Nica} gives Nica covariance of $(b^r,w)$ for all $w$. Thus $S$ is Nica covariant.

The universal property of $(C^*(G,P),T)$ now gives us the homomorphism $\pi_{U,V}:=\pi_S:C^*(G,P)\to B$ with the required properties.
\end{proof}

\section{A characterisation of KMS states}\label{sect:char}

The height map $\theta$ gives  a strongly continuous gauge action $\gamma:\T\to \Aut C^*(G,P)$ such that $\gamma_z(T_x)= z^{\theta(x)}T_x$. We then define $\alpha:\R\to \Aut C^*(G,P)$ by $\alpha_t= \gamma_{e^{i t}}$, and aim to study the KMS states of the dynamical system $(C^*(G,P), \alpha)$.  
For $x,y\in P$  we have $\alpha_t(T_xT_y^*) = e^{i t(\theta(x) - \theta(y))}T_xT_y^*$, and thus each $T_xT_y^*$ is analytic, with 
$
\alpha_z(T_xT_y^*)=e^{i z(\theta(x) - \theta(y))}T_xT_y^*
$.
Since the $T_xT_y^*$ span a dense subspace of $C^*(G,P)$, it follows from \cite[Proposition~8.12.3]{P} that a state $\psi$ of $C^*(G,P)$ is a KMS$_\beta$ state of $(C^*(G,P),\alpha)$ for some $\beta\in \R\setminus\{0\}$ if and only if
\begin{equation}\label{eq:KMScondition}
\psi((T_xT_y^*)(T_pT_q^*))=\psi((T_pT_q^*)\alpha_{i\beta}(T_xT_y^*))=e^{-\beta(\theta(x)-\theta(y))}\psi((T_pT_q^*)(T_xT_y^*))
\end{equation}
for all $x,y,p,q\in P$.

\begin{prop}\label{charKMS}
Let $\psi$ be a state on $(C^*(G,P),\alpha)$.  Then $\psi$ is a KMS$_\beta$ state if and only if for all $x,y\in P$ we have
\begin{equation}
 \label{eq:KMS}
\psi(T_xT_y^*) = \begin{cases}
                  e^{-\beta \theta(x)}\psi(T_{y^{-1}x}) & \text{if  $\theta(x)=\theta(y)$ and  $x \vee y = x$}\\
		  e^{-\beta \theta(x)}\psi(T_{x^{-1}y}^*) & \text{if  $\theta(x)=\theta(y)$ and  $x \vee y = y$}\\
		  0 & \text{otherwise.}
                 \end{cases}
\end{equation}
\end{prop}

\begin{proof}
 Suppose $\psi$ is a KMS$_\beta$ state on $(C^*(G,P),\alpha)$ and fix $x,y\in P$. Nica covariance of $T$ gives
\[
T_y^*T_x =  
\begin{cases}
T_{y^{-1}(y\vee x)}T^*_{x^{-1}(y\vee x)} &\text{if $x\vee y<\infty$}\\
0&\text{if $x\vee y=\infty$.}
\end{cases}\]
The KMS condition says
\begin{equation*}
%\label{eq1}
 \psi(T_xT_y^*) =  e^{-\beta \theta(x)} \psi(T_y^*T_x), 
\end{equation*}
and hence $ \psi(T_xT_y^*)= 0$ unless $x\vee y < \infty$.
Applying the KMS condition again gives
\[
  \psi(T_xT_y^*) = e^{-\beta (\theta(x)-\theta(y))}\psi(T_xT_y^*),
\] 
and hence also $\psi(T_xT_y^*) =0$ unless  $\theta(x) = \theta(y)$. 

Now suppose that $\theta(x) = \theta(y)$ and $x\vee y < \infty$.
Then 
\[
\psi(T_xT_y^*) =  e^{-\beta \theta(x)} \psi(T_{y^{-1}(x\vee y)} T_{x^{-1}(x\vee y)}^*),
\]
and we recover~(\ref{eq:KMS}) since either $x\vee y = x$ or $x\vee y = y$ by Lemma~\ref{lem:iains}(\ref{it:iains2}).

Conversely, suppose $\psi$ is a state satisfying \eqref{eq:KMS}. We fix $x,y,p,q\in P$ and aim to show the KMS condition \eqref{eq:KMScondition} holds.   We will show that if $\psi(T_xT_y^*T_pT_q^*)\neq 0$, then $\psi (T_pT_q^*T_xT_y^*)\neq 0$ also and the KMS condition holds. Then, by symmetry, $\psi(T_xT_y^*T_pT_q^*)\neq 0$ if and only if $\psi (T_pT_q^*T_xT_y^*)\neq 0$, and so if $\psi(T_xT_y^*T_pT_q^*)=0$ then the KMS condition holds with both sides zero. So we  assume that $\psi(T_xT_y^*T_pT_q^*) \neq 0$.

By Nica covariance $y\vee p<\infty$ and  
\begin{equation}\label{eq-usedcovariance}
0\neq \psi(T_xT_y^*T_pT_q^*)= \psi(T_{xy^{-1}(y \vee p)} T^*_{qp^{-1}(y \vee p)}).
\end{equation}
We now argue that it suffices to show the KMS condition when  $\theta(y)\geq \theta(p)$ and $y\vee p=yb^m$ for some $m\in\N$.  If $\theta(y)> \theta(p)$, then there exists $m\in\N$  such that $y\vee p=yb^m$ by Lemma~\ref{lem:iains}\eqref{it:iains1}.  If $\theta(y)=\theta(p)$, then there exists $m\in\N$  such that  $y\vee p=yb^m$   by Lemma~\ref{lem:iains}\eqref{it:iains2} ($m=0$ is allowed). If $\theta(y)< \theta(p)$, then we take the adjoint of $T_xT_y^*T_pT_q^*$ and use that $\psi(a^*)=\overline{\psi(a)}$.
So we assume that $\theta(y)\geq \theta(p)$ and $y\vee p=yb^m$ for some $m\in\N$.

Set \[M:= xy^{-1}(y \vee p)=xy^{-1}yb^m=xb^m\quad\text{and}\quad N:= qp^{-1}(y \vee p)=qp^{-1}yb^m.\]
Then \eqref{eq-usedcovariance} and the equation for $\psi$  at \eqref{eq:KMS} implies that   $\theta(M)=\theta(N)$, and either $M\vee N=M$ or $M\vee N=N$. Thus 
\[
\theta(x)-\theta(y)=\theta(q)-\theta(p),%\quad\text{and}\quad\theta(N)=\theta(x),
\]
and then $\theta(x)\geq \theta(q)$. By Lemma~\ref{lem:iains}\eqref{it:iains2}  there exists $n \in \Z$ such that $M = Nb^n$. For future use we note here that $M = Nb^n$ implies
\begin{equation}\label{eq-x}
x=qp^{-1}yb^n.
\end{equation}
Using \eqref{eq:KMS} we have
\begin{equation} \label{eq:tobereconcile}
\psi(T_xT_y^*T_pT_q^*)= \psi(T_{M}T_N^*) 
=\begin{cases}
e^{-\beta \theta(N)} \psi(T_{b^n}) &\text{if $n\geq 0$}\\
e^{-\beta \theta(N)} \psi(T_{b^{-n}}^*) &\text{if $n<0$.}
\end{cases} 
\end{equation}

Next we consider $\psi(T_pT_q^*T_xT_y^*)$. Since $x\leq M$ and $q\leq N$ we have $x\vee q\leq M\vee N<\infty$, and $\psi(T_pT_q^*T_xT_y^*)=\psi(T_{pq^{-1}(x\vee q)}T_{yx^{-1}(x\vee q)}^*)$.  By Lemma~\ref{lem:iains}, there exists $s\in\N$ such that either $x\vee q=xb^s$ or $x\vee q=qb^s$.
First, suppose that $x\vee q=xb^s$. Then 
\begin{gather*}
yx^{-1}(x\vee q)=yb^s\text{\ and}\\
pq^{-1}(x\vee q)=pq^{-1}xb^s=yb^nb^s= yx^{-1}(x\vee q)b^n
\end{gather*}
using \eqref{eq-x}.
Second, suppose that $x\vee q=qb^s$. Since $\theta(x)\geq \theta(q)$, we have  $x=x\vee q$. Thus
\begin{gather*}
pq^{-1}(x\vee q)=pb^s\text{\ and}\\
yx^{-1}(x\vee q)=y=pq^{-1}xb^{-n}=pb^sb^{-n}=pq^{-1}(x\vee q)b^{-n}
\end{gather*}
using \eqref{eq-x}. In either case, $pq^{-1}(x\vee q)= yx^{-1}(x\vee q)b^n$, and 
\begin{align*}
\psi(T_pT_q^*T_xT_y^*) = \psi(T_{pq^{-1}(q \vee x)}T_{yx^{-1}(q \vee x)}^*)=\begin{cases}
e^{-\beta \theta(pq^{-1}(x\vee q))} \psi(T_{b^n}) &\text{if $n\geq 0$}\\
e^{-\beta \theta(pq^{-1}(x\vee q))} \psi(T_{b^{-n}}^*) &\text{if $n<0$.}
\end{cases} \end{align*}
But $\theta(pq^{-1}(x\vee q))=\theta(p)-\theta(q)+\theta(x)=\theta(y)$ and $\theta(N)=\theta(x)$. Thus
\[
e^{-\beta(\theta(x)-\theta(y))}\psi(T_pT_q^*T_xT_y^*) =\begin{cases}
e^{-\beta\theta(x)} \psi(T_{b^n}) &\text{if $n\geq 0$}\\
e^{-\beta\theta(x)} \psi(T_{b^{-n}}^*) &\text{if $n<0$}
\end{cases} 
\]
is the same as \eqref{eq:tobereconcile}, as required.
As we argued above, this suffices to show that $\psi$ is a KMS$_\beta$ state.
\end{proof}

\begin{cor}\label{btsuffices}
Suppose that $\phi$ and $\psi$ are KMS$_\beta$ states on $(C^*(G,P),\alpha)$ and $\phi(T_{b^t})=\psi(T_{b^t})$ for all $t\in \N$. Then $\phi=\psi$.
\end{cor}

\begin{proof}
Both states vanish on generators $T_xT_y^*$ unless $\theta(x)=\theta(y)$ and $x\vee y$ is $x$ or $y$, in which case Lemma~\ref{lem:iains} implies that either $y^{-1}x$ or $x^{-1}y$ has the form $b^t$. Thus $\phi(T_xT_y^*)=\psi(T_xT_y^*)$ for all $x,y\in P$, and $\phi=\psi$.
\end{proof}

\begin{cor}\label{restonbeta}
Consider the dynamical system $(C^*(G,P),\alpha)$ as above and take $\beta\in \R$. 
\begin{enumerate}
\item\label{factor1} Every KMS$_\beta$ state of $(C^*(G,P),\alpha)$ factors through the quotient by the ideal generated by $1-T_bT_b^*$.
\item\label{lowerbd} If $\beta<\ln d$, then $(C^*(G,P),\alpha)$ has no KMS$_\beta$ states.
\item\label{factor2} Let $I$ be the ideal generated by the element
\[
1-\sum_{j=0}^{d-1}T_{b^ja}T_{b^ja}^*.
\]
Then a KMS$_{\beta}$ state factors through the quotient $\OO(G,P):=C^*(G,P)/I$ if and only if $\beta=\ln d$.
 \end{enumerate}
\end{cor}

\begin{proof}
For \eqref{factor1}, suppose that $\psi$ is a KMS$_\beta$ state of $(C^*(G,P),\alpha)$. Then
\[
\psi(T_bT_b^*)=\psi(T_b^*\alpha_{i\beta}(T_b))=\psi(T_b^*T_b)=\psi(1)=1.
\]
Thus $\psi(1-T_bT_b^*)=0$. The projection $1-T_bT_b^*$ is invariant for the dynamics, and the elements $T_xT_y^*$ are analytic elements such that $\alpha_z(T_xT_y^*)$ is the product of $T_xT_y^*$ by the scalar-valued function $z\mapsto e^{iz(\theta(x)-\theta(y))}$. So we apply Lemma 2.2 of \cite{aHLRS} with $P=\{1-T_bT_b^*\}$ and $\FF=\{T_xT_y^*\}$, and deduce that $\psi$ factors through a state of the quotient, as claimed.

For \eqref{lowerbd}, we again suppose that $\psi$ is a KMS$_\beta$ state of $(C^*(G,P),\alpha)$. Then since \[\{T_{b^ja}:0\leq j<d\}\] is a Toeplitz-Cuntz family, we have $1\geq\sum_{j=0}^{d-1}T_{b^ja}T_{b^ja}^*$ in $C^*(G,P)$. Thus \eqref{eq:KMS} gives
\begin{equation}\label{applyTCK}
1=\psi(1)\geq \sum_{j=0}^{d-1}\psi(T_{b^ja}T_{b^ja}^*)=\sum_{j=0}^{d-1}e^{-\beta}\psi(1)=e^{-\beta}d,
\end{equation}
which is equivalent to $\beta\geq \ln d$.

We now prove \eqref{factor2}. If $\psi$ is a KMS$_{\ln d}$ state, then $1=e^{-\beta}d$ forces equality throughout \eqref{applyTCK}, and 
\begin{equation}\label{Cuntzrel}
\psi\Big(1-\sum_{j=0}^{d-1}T_{b^ja}T_{b^ja}^*\Big)=0.
\end{equation}
Now another application of \cite[Lemma~2.2]{aHLRS} shows that $\psi$ factors through the quotient $\OO(G,P)$. Conversely, if $\psi$ is a KMS$_\beta$ state which factors through the quotient, then $\psi$ satisfies \eqref{Cuntzrel}, we have equality in \eqref{applyTCK}, and $\beta=\ln d$.
\end{proof}

\section{KMS states for large inverse temperatures}\label{sect:large}

\begin{thm}\label{KMSToe}
Suppose that $\beta>\ln d$ and $\mu$ is a probability measure on $\T$. Then there is a KMS$_\beta$ state $\psi_{\beta,\mu}$ on $(C^*(G,P),\alpha)$ such that
\begin{equation}\label{psivsmoments}
\psi_{\beta,\mu}(T_{b^t})=(1-e^{-\beta}d)\Big(\int_{\T}z^t\,d\mu(z)+\sum_{\{k\geq 1\,:d\,\mid\, c^jd^{-j}t\text{ for }0\leq j<k\}}e^{-\beta k}d^k\int_{\T}z^{c^kd^{-k}t}\,d\mu(z)\Big),
\end{equation}
where the sum is interpreted as $0$ if there are no such integers $k$. Every KMS$_\beta$ state has this form. If $d$ does not divide $c$, then the map $\mu\mapsto \psi_{\beta,\mu}$ is an affine continuous isomorphism  of the simplex $P(\T)$ of probability measures on $\T$ onto the KMS$_\beta$ simplex of $(C^*(G,P),\alpha)$.
\end{thm}

The proof of this theorem will occupy the rest of this section. Our first task is to show existence of such states, and for this we need some concrete representations of $C^*(G,P)$.

We consider the subgroup $K:=\{b^t:t\in \Z\}$ of the Baumslag-Solitar group $G=BS(c,d)$, and let $W:K\to U(H)$ be a unitary representation. We choose a section $c:G/K\to G$ for the quotient map, and write $c(g)$ for $c(gK)$. Then we can realise the induced representation $\Ind_K^GW$ as acting in the space $\ell^2(G/K,H)$ according to the formula
\[
\big((\Ind_K^GW)_l\xi\big)(gK)=W_{c(g)^{-1}lc(l^{-1}g)}\big(\xi(l^{-1}gK)\big)\quad\text{for $\l\in G$, $\xi\in \ell^2(G/K,H)$}.
\]
(See, for example, \cite[page~50]{KT}.) 

\begin{prop}\label{indrepconst}
Choose a section $c:G/K\to G$ such that $c(xK)=\stem (x)$ for every $x\in P$, and use $c$ to pull over the usual orthonormal basis $\{e_{k,\sigma}:\sigma\in\Sigma_k\}$ for $\ell^2(\Sigma_k)$ to an orthonormal set in $\ell^2(G/K)$. Then the subspace
\[
H_0:=\clsp\{e_{k,\sigma}\otimes h:k\in \N,\ \sigma\in \Sigma_k, h\in H\}
\]
of $\ell^2(G/K,H)$ is invariant for every $(\Ind_K^GW)_x$ with $x\in P$, and we then have
\begin{equation}\label{formpi(Tx)}
(\Ind_K^GW)_x(e_{k,\sigma}\otimes h)=e_{k+\theta(x),\stem(x\sigma)}\otimes W_{b^t}h\quad\text{where $x\sigma=\stem(x\sigma)b^t$.}
\end{equation}
The map $x\mapsto (\Ind_K^GW)_x\big|_{H_0}$ is a Nica-covariant isometric representation of $P$, and the corresponding representation $\pi$ of $C^*(G,P)$ on $H_0$ satisfies
\begin{equation}\label{formforpi}
\pi(T_x)(e_{k,\sigma}\otimes h)=e_{k+\theta(x),\stem(x\sigma)}\otimes W_{b^t}h\quad\text{where $x\sigma=\stem(x\sigma)b^t$.}
\end{equation}
The operator $\pi(T_b)$ is unitary.
\end{prop}

\begin{proof}
It suffices to check invariance for the generators $x=a$ and $x=b$ of $P$. First we take $x=a$. We are viewing functions in $\ell^2(\Sigma_k)$ as functions on $G/K$ by viewing $\Sigma_k$ as the subset $\{\sigma K:\sigma\in \Sigma_k\}$ of $G/K$ and extending the functions to be $0$ off $\Sigma_k$. Thus
\[
e_{k,\sigma}(gK)=\begin{cases}
1&\text{if $gK=\sigma K$}\\
0&\text{otherwise},
\end{cases}
\]
and
\begin{align*}
(\Ind_K^GW)_a(e_{k,\sigma}\otimes h)(gK)
&=W_{c(g)^{-1}ac(a^{-1}g)}(e_{k,\sigma}(a^{-1}gK)h)\\
&=\begin{cases}
W_{c(g)^{-1}ac(a^{-1}g)}h&\text{if $a^{-1}gK=\sigma K$}\\
0&\text{otherwise}.
\end{cases}
\end{align*}
Now we have
\[
a^{-1}gK=\sigma K\Longleftrightarrow a^{-1}g=\sigma b^n\text{ for some $n$}\Longleftrightarrow g=a\sigma b^n\text{ for some $n$;}
\]
then, since $a\sigma$ is a stem, and $c(\tau K)=\tau$ for $\tau\in \bigcup_k\Sigma_k$, we have $c(g)^{-1}ac(a^{-1}g)=(a\sigma)^{-1}a\sigma=e$. So 
\begin{align}\label{formindWa}
(\Ind_K^GW)_a(e_{k,\sigma}\otimes h)(gK)&=\begin{cases}
h&\text{if $gK=a\sigma K$}\\
0&\text{otherwise}
\end{cases}\\
&=(e_{k+1,a\sigma}\otimes h)(gK).\notag
\end{align}
This gives invariance of $H_0$ for $(\Ind_K^GW)_a$. 

For $x=b$, similar considerations give
\[
(\Ind_K^GW)_b(e_{k,\sigma}\otimes h)(gK)
=\begin{cases}
W_{c(g)^{-1}bc(b^{-1}g)}h&\text{if $b^{-1}gK=\sigma K$}\\
0&\text{otherwise,}
\end{cases}
\]
and 
\[
b^{-1}gK=\sigma K\Longleftrightarrow b^{-1}g=\sigma b^n\text{ for some $n$}\Longleftrightarrow g=b\sigma b^n\text{ for some $n$.}
\]
While $b\sigma$ need not be a stem, it is certainly in $P$, and hence $c(b\sigma K)=\stem(b\sigma)$. Then $b\sigma=\stem(b\sigma)b^t$ for some $t\in \N$, and
\[
c(g)^{-1}bc(b^{-1}g)=(\stem(b\sigma))^{-1}b\sigma=b^t.
\]
Thus
\begin{equation}\label{formindWb}
(\Ind_K^GW)_b(e_{k,\sigma}\otimes h)
=e_{k,\stem(b\sigma)}\otimes W_{b^t}h,
\end{equation}
which is back in $\ell^2(\Sigma_k)\otimes H\subset H_0$.

Since both $U:=(\Ind_K^GW)_b|_{H_0}$ and $V:=(\Ind_K^GW)_a|_{H_0}$ are the restrictions of unitary operators, they are isometries.  We next prove \eqref{formpi(Tx)}. First, we prove it by induction on $i$ for $x$ of the form $b^i$. It is by definition true for $i=1$. If it is true for $i$, then
\begin{align*}
(\Ind_K^GW)_{b^{i+1}}&(e_{k,\sigma}\otimes h)\\
&=(\Ind_K^GW)_{b^{i}}(e_{k,\stem(b\sigma)}\otimes W_{b^t}h)\quad\text{where $b\sigma=\stem(b\sigma)b^t$}\\
&=e_{k,\stem(b^i\stem(b\sigma))}\otimes W_{b^s}W_{b^t}h\quad\text{where $b^i\stem(b\sigma)=\stem(b^i\stem(b\sigma))b^s$.}
\end{align*}
Lemma~\ref{forlisa} implies that $\stem(b^{i+1}\sigma)=\stem(b^i\stem(b\sigma))$ and $b^{i+1}\sigma=\stem(b^{i+1}\sigma)b^{s+t}$, and we have proved the result for $x=b^i$. It is quite easy to check that it works for $x=b^ia$, and then a similar argument by induction on the height of $x$ (again using Lemma~\ref{forlisa}) gives  the general result.

We aim to prove that $U$ and $V$ satisfy the relations of Proposition~\ref{defrel}. Since we know from Lemma~\ref{maponstems} that $\sigma\mapsto \stem(b\sigma)$ is a bijection, Equation \eqref{formindWb} implies that $U$ is surjective. Thus $U$ is unitary, and it suffices to verify that $U$ and $V$ satisfy relations \eqref{t1} and \eqref{t5} of Proposition~\ref{defrel}.

For \eqref{t1} we have on the one hand
\begin{align*}
VU^c(e_{k,\sigma}\otimes h)&=V(e_{k,\stem(b^c\sigma)}\otimes W_{b^t}h)\quad\text{where $b^c\sigma=\stem(b^c\sigma)b^t$}\\
&=e_{k+1,a\stem(b^c\sigma)}\otimes W_{b^t}h\quad\text{where $b^c\sigma=\stem(b^c\sigma)b^t$;}
&\end{align*}
on the other hand
\[
U^dV(e_{k,\sigma}\otimes h)=e_{k+1,\stem(b^da\sigma)}\otimes W_{b^s}h\quad\text{where $b^da\sigma=\stem(b^da\sigma)b^s$.}
\]
Now Lemma~\ref{forlisa} gives $a\stem(b^c\sigma)=\stem(ab^c\sigma)=\stem (b^da\sigma)$, and 
\[
b^t=\stem(b\sigma)^{-1}b^c\sigma=(a\stem(b\sigma))^{-1}ab^c=\stem(b^da\sigma)^{-1}b^da\sigma=b^s,
\]
and we have \eqref{t1}. For \eqref{t5}, we observe that 
$U^jV(e_{k,\sigma}\otimes h)$ has the form $e_{k+1,\stem(b^ja\sigma)}\otimes W_{b^t}h$. For $0\leq j<d$ and $\sigma\in \Sigma_k$, $b^ja\sigma$ is itself a stem, and for $1\leq j<d$ it is distinct from $a\sigma$. Thus for $1\leq j<d$ the vector $U^jV(e_{k,\sigma}\otimes h)$ is always orthogonal to $V(e_{k,\sigma}\otimes h)=e_{k+1,a\sigma}\otimes h$, and hence to the range of $V$. Thus $V^*U^jV(e_{k,\sigma}\otimes h)$ is always $0$, and we have~\eqref{t5}. 

Now Proposition~\ref{defrel} implies that $S:x\mapsto (\Ind_K^GW)_x\big|_{H_0}$ is Nica covariant, and gives a homomorphism $\pi=\pi_S$ such that $\pi(T_{b^t})=S_{b^t}=U^t$ is unitary.
\end{proof} 

\begin{prop}\label{abstractKMSconst}
In the situation of Proposition~\ref{indrepconst}, fix a unit vector $h\in H$. Then for every $\beta>\ln d$ there is a KMS$_\beta$ state $\psi_h$ on $(C^*(G,P),\alpha)$ such that
\begin{equation}\label{defphih}
\psi_h(a)=(1-e^{-\beta}d)\sum_{k=0}^\infty\sum_{\sigma\in\Sigma_k}e^{-\beta k}\big(\pi(a)(e_{k,\sigma}\otimes h)\,|\,e_{k,\sigma}\otimes h\big)\quad\text{for $a\in C^*(G,P)$.}
\end{equation}
\end{prop}

\begin{proof}
We begin by checking that the series converges. Indeed, since $|\Sigma_k|=d^k$, and $e^\beta>d$, we have
\[
\sum_{k=0}^\infty\sum_{\sigma\in\Sigma_k}e^{-\beta k}\big(e_{k,\sigma}\otimes h\,|\,e_{k,\sigma}\otimes h\big)=\sum_{k=0}^\infty e^{-\beta k}d^k=\frac{1}{1-e^{-\beta}d}.
\]
In particular, $\psi_h(1)=1$, and we have a well-defined state.

We now want to verify that $\psi_h$ satisfies Equation~\eqref{eq:KMS} in Proposition~\ref{charKMS}. So we take $x,y\in P$ and consider 
\begin{equation}\label{singleip}
\big(\pi(T_xT_y^*)(e_{k,\sigma}\otimes h)\,|\,e_{k,\sigma}\otimes h\big)
=\big(\pi(T_y^*)(e_{k,\sigma}\otimes h)\,|\,\pi(T_x^*)(e_{k,\sigma}\otimes h)\big).
\end{equation}
For each $k\geq 0$, we identify
\[
\clsp\{e_{k,\sigma}\otimes g:\sigma\in \Sigma_k, g\in H\}
\]
with $\ell^2(\Sigma_k)\otimes H$. Then the subspaces $\{\ell^2(\Sigma_k)\otimes H:k\in \N\}$ are mutually orthogonal, and we have 
\[
H_0=\textstyle{\bigoplus_{k=0}^{\infty}\ell^2(\Sigma_k)\otimes H}.
\]
The operator $\pi(T_x)$ maps each summand $\ell^2(\Sigma_k)\otimes H$ into $\ell^2(\Sigma_{k+\theta(x)})\otimes H$, and hence the adjoint $T_x^*$ vanishes on $\ell^2(\Sigma_k)\otimes H$ for $k<\theta(x)$, and maps the other $\ell^2(\Sigma_k)\otimes H$ into $\ell^2(\Sigma_{k-\theta(x)})\otimes H$. Thus when $\theta(x)\not=\theta(y)$,  $T_x^*$ and $T_y^*$ map $\ell^2(\Sigma_k)\otimes H$ into orthogonal summands in $H_0=\bigoplus \ell^2(\Sigma_k)\otimes H$. Thus if $\theta(x)\not=\theta(y)$, we have $\psi_h(T_xT_y^*)=0$. 

It remains to consider $x,y$ satisfying $\theta(x)=\theta(y)$. Then by Lemma~\ref{lem:iains} we have one of $x\vee y=x$, $x\vee y=y$ or $x\vee y=\infty$. If $x\vee y=\infty$, then Nica covariance of $T$ implies that $T_x^* T_y=0$, so that the range of $\pi(T_x)$ is orthogonal to the range of $\pi(T_y)$; since $\pi(T_x)^*(e_{k,\sigma}\otimes h)=0$ unless $e_{k,\sigma}\otimes h$ is in the range of $\pi(T_x)$, all the inner products \eqref{singleip} are $0$, and $\psi(T_xT_y^*)=0$. Since 
\[
\big(\pi(T_yT_x^*)(e_{k,\sigma}\otimes h)\,|\,e_{k,\sigma}\otimes h\big)=
\overline{\big(\pi(T_xT_y^*)(e_{k,\sigma}\otimes h)\,|\,e_{k,\sigma}\otimes h\big)},
\]
it remains for us to compute $\psi(T_xT_y^*)$ when $x\vee y=x$ (if $x\vee y=y$ switch $x$ and $y$). So suppose $x\vee y=x$. Then Lemma~\ref{lem:iains} implies that $x=yb^t$ for some $t\in \N$.

We begin by fixing $\sigma\in \Sigma_k$ and computing
\begin{align}\label{firstcalc}
\big(\pi(T_xT_y^*)(e_{k,\sigma}\otimes h)\,|\,e_{k,\sigma}\otimes h\big)
&=\big(\pi(T_yT_{b^t}T_y^*)(e_{k,\sigma}\otimes h)\,|\,e_{k,\sigma}\otimes h\big)\\
&=\big(\pi(T_{b^t})\pi(T_y^*)(e_{k,\sigma}\otimes h)\,|\,\pi(T_y^*)(e_{k,\sigma}\otimes h)\big).
\notag
\end{align}

Notice that because $\pi(T_b)$ is unitary, the operator $\pi(T_xT_y^*)$ is not changed if we replace $y$ by its stem. So we assume that $y$ is a stem. Then  $y\sigma$ is also a stem, so the formula \eqref{formforpi} implies that $\pi(T_y)(e_{k,\sigma}\otimes h)=e_{k+\theta(y),y\sigma}\otimes h$. Since each $e_{k,\sigma}\otimes h $ is either in the range of $\pi(T_y)$ or orthogonal to it, $\pi(T_y)^*(e_{k,\sigma}\otimes h)$ vanishes unless $k\geq \theta(y)$ and $\sigma$ has the form $y\tau$ for some $\tau\in \Sigma_{k-\theta(y)}$. Then
\[
\big(\pi(T_xT_y^*)(e_{k,\sigma}\otimes h)\,|\,e_{k,\sigma}\otimes h\big)
=\big(\pi(T_{b^t})(e_{k-\theta(y),\tau}\otimes h)\,|\,e_{k-\theta(y),\tau}\otimes h\big).
\]
Next we observe that because $y$ is a stem, $\tau\mapsto y\tau$ is an injection of $\Sigma_j$ into $\Sigma_{j+\theta(y)}$ for every $j\geq 0$. Thus
\begin{align*}
\psi_h(T_x T_y^*)&=(1-e^{-\beta}d)\sum_{k=\theta(y)}^\infty\sum_{\tau\in \Sigma_{k-\theta(y)}}e^{-\beta k}\big(\pi(T_{b^t})(e_{k-\theta(y),\tau}\otimes h)\,|\,e_{k-\theta(y),\tau}\otimes h\big)\\
&=(1-e^{-\beta}d)e^{-\beta\theta(y)}\sum_{j=0}^\infty\sum_{\tau\in \Sigma_{j}}e^{-\beta j}\big(\pi(T_{b^t})(e_{j,\tau}\otimes h)\,|\,e_{j,\tau}\otimes h\big)\\
&=e^{-\beta\theta(y)}\psi_h(T_{b^t})\\
&=e^{-\beta\theta(y)}\psi_h(T_{y^{-1}x}).
\end{align*}
Thus $\psi_h$ satisfies \eqref{eq:KMS}, and Proposition~\ref{charKMS} implies that $\psi_h$ is a KMS$_\beta$ state.
\end{proof}

The subgroup $K$ is a copy of the additive group $\Z$, written multiplicatively because it sits inside the nonabelian group $G$. Thus $C^*(K)$ is isomorphic to the algebra $C(\T)$, and states on $C^*(K)$ are given by probability measures $\mu$ on $\T$. For such a measure $\mu$, we consider the representation $W=W(\mu)$ of $K$ on $H=L^2(\T,d\mu)$ given by 
\begin{equation}\label{defW}
(W_{b^t}f)(z)=z^tf(z),
\end{equation}
and the unit vector $h=1_\mu$ in $L^2(\T,d\mu)$ associated to the constant function $1$. Then Proposition~\ref{abstractKMSconst} gives us a KMS$_\beta$ state $\psi_{\beta,\mu}:=\psi_{1_\mu}$, and we need to calculate the values of this state on the elements $T_{b^t}$, which by Corollary~\ref{btsuffices} determine the state. For $k=0$, we have just the trivial stem $e$, and 
\begin{equation}\label{k=0}
\big(\pi(T_{b^t})(e_{k,e}\otimes 1_\mu)\,\big|\,e_{k,e}\otimes 1_\mu\big)=(W_{b^t}1_\mu\,|\,1_\mu)=\int_{\T} z^t\,d\mu(z).
\end{equation}
For $k\geq 1$ and each stem $\sigma\in \Sigma_k$, we have 
\begin{equation}\label{calcpiT}
\pi(T_{b^t})(e_{k,\sigma}\otimes 1_\mu)=e_{k,\stem(b^t\sigma)}\otimes W_{b^s}1_\mu\quad\text{where $b^t\sigma=\stem(b^t\sigma)b^s$.}
\end{equation}
Thus 
\[
\big(\pi(T_{b^t})(e_{k,\sigma}\otimes 1_\mu)\,\big|\,e_{k,\sigma}\otimes 1_\mu\big)
\]
vanishes unless $\stem(b^t\sigma)=\sigma$, and hence we need to know when this happens.

\begin{lemma}\label{idsummands}
Suppose that $\sigma\in \Sigma_k$ for some $k\geq 1$ and $b^t\in K$. Then $\stem(b^t\sigma)=\sigma$ if and only if $d$ divides $c^jd^{-j}t$ for every $j$ such that $1\leq j<k$. If so, we have
\begin{equation}\label{carrythrough}
b^t\sigma =\sigma b^{c^kd^{-k}t}.
\end{equation}
\end{lemma}

\begin{proof}
Suppose first that $\stem(b^t\sigma)=\sigma$, and write $\sigma=b^{t_0}ab^{t_1}a\cdots b^{t_{k-1}}a$ in normal form. Then since $\stem(b^t\sigma)$ begins with $\stem(b^tb^{t_0}a)$, we have $\stem(b^tb^{t_0}a)=b^{t_{0}}a$. Write $t+t_0=r+nd$ with $0\leq r<d$, and then $b^ra=\stem(b^tb^{t_0}a)=b^{t_0}a$. Now uniqueness of the normal form gives $r=t_0$, $t=nd$  and $b^tb^{t_0}a=b^{t_0}ab^{nc}=b^{t_0}ab^{cd^{-1}t}$. Now running a similar argument on $b^{cd^{-1}t}b^{t_1}a$ shows that $d$ divides $cd^{-1}t$, and $b^{cd^{-1}t}b^{t_1}a=b^{t_1}ab^{c^2d^{-2}t}$. Continuing this way shows that $d$ divides $c^jd^{-j}t$ for every $j<k$, and gives the formula \eqref{carrythrough}.

For the converse, just note that the condition on $d$ allows us to pull $b^t$  through $\sigma$ without changing the powers $b^{t_j}$ for $j\leq k$.
\end{proof}

\begin{rmk}
When $c$ and $d$ are coprime, $d$ divides $c^jd^{-j}t$ if and only if $d^{j+1}$ divides $t$, and hence left multiplication by  $b^t$ fixes all the stems in $\Sigma_k$ if and only if $d^k$ divides $t$. However, in general it is possible that $d$ divides $c^jd^{-j}t$ but does not divide $c^{j-1}d^{-(j-1)}$. For example, consider $c=8$ and $d=12$. Then $cd^{-1}t=2t/3$ and $c^2d^{-2}t=4t/9$. Thus with $t=27$, $d=12$ divides $c^2d^{-2}t$ but not $cd^{-1}t$.
\end{rmk}

Equation~\eqref{k=0} tells us where the first integral in \eqref{psivsmoments} comes from. For $k\geq 1$, Lemma~\ref{idsummands} tells us why only the summands described in \eqref{psivsmoments} survive. Suppose that $k\in \N$ and $d$ divides $c^jd^{-j}t$ whenever $1\leq j<k$. Then \eqref{carrythrough} tells us that the $s$ in \eqref{calcpiT} is $c^kd^{-k}t$. Thus the $k$th summand in the formula \eqref{defphih} for $\psi_{\beta,\mu}(T_{b^t})=\phi_{1_\mu}(T_{b^t})$ is
\begin{align*}
(1-e^{-\beta}d)\sum_{\sigma\in \Sigma_k}e^{-\beta k}&\big(\pi(T_{b^t})(e_{k,\sigma}\otimes 1_\mu)\,\big|\,e_{k,\sigma}\otimes 1_\mu\big)\\
&=(1-e^{-\beta}d)\sum_{\sigma\in \Sigma_k}e^{-\beta k}\big(e_{k,\sigma}\otimes W_{b^{c^kd^{-k}t}}1_\mu)\,\big|\,e_{k,\sigma}\otimes 1_\mu\big)\\
&=(1-e^{-\beta}d)\sum_{\sigma\in \Sigma_k}e^{-\beta k}\int_{\T} z^{c^kd^{-k}t}\,d\mu(z)\\
&=(1-e^{-\beta}d)|\Sigma_k|e^{-\beta k}\int_{\T} z^{c^kd^{-k}t}\,d\mu(z)\\
&=(1-e^{-\beta}d)d^ke^{-\beta k}\int_{\T} z^{c^kd^{-k}t}\,d\mu(z).
\end{align*}
Thus we recover the formula \eqref{psivsmoments} for $\psi_{\beta,\mu}(T_{b^t})$.

Next we have to prove that every KMS$_\beta$ state has the form $\psi_{\beta,\mu}$. So we suppose that $\phi$ is a KMS$_\beta$ state. Recall that the set
\[
\{T_x:x\in \Sigma_1\}=\{U^iV: 0\leq i<d\}
\]
is a Toeplitz-Cuntz-Krieger family, and set
\[
P:=1-\sum_{x\in \Sigma_1} T_xT_x^*.
\]
Then the KMS condition implies that $\phi(P)=1-e^{-\beta}d$, so we may consider the \emph{conditioned state} 
\[
\phi_P:a\mapsto (1-e^{-\beta}d)^{-1}\phi(PaP).
\]
This is in particular a state on the $C^*$-subalgebra $C^*(K)=C^*(T_b)\cong C(\T)$, and hence is given by a measure $\mu$ on $\T$: we choose the measure that satisfies
\[
\phi_P(T_{b^n})=\int_{\T} z^n\,d\mu(z)\quad\text{ for $n\in \N$.}
\]
We aim to prove that $\phi=\psi_{\beta,\mu}$. The argument follows closely that of \cite[Proposition~7.1]{LRR}, though the details in the calculations are quite different.

We begin by claiming that 
\[
\big\{T_xPT_x^*:x\in \textstyle{\bigcup_{k=0}^{\infty}}\Sigma_k\big\}
\]
is a family of mutually orthogonal projections. To see this, take $x\in \Sigma_j$ and $y\in \Sigma_k$ and consider $PT_x^*T_yP$. If $j=k$, then either $x=y$ or $x\vee y=\infty$, and Nica covariance gives $PT_x^*T_yP=\delta_{x,y}P$. If $j\not=k$, then Nica covariance leaves a factor of the form $PT_\sigma$ or $T_\sigma^* P$ with $\theta(\sigma)>0$. Now we write $\sigma=w\sigma'$ with $w\in \Sigma_1$, and then
\begin{align*}
PT_\sigma&=\big(1-\sum_{z\in \Sigma_1} T_zT_z^*\Big)T_wT_{\sigma'}\\
&=\big(T_w-\sum_{z\in \Sigma_1} T_zT_z^*T_w\Big)T_{\sigma'}\\
&=(T_w-T_w)T_{\sigma'}=0.
\end{align*} 
Thus $PT_x^*T_yP=0$ when $j\neq k$, and the claim follows.

Now for each $n$,  
\[
P_n:=\sum_{k=0}^n\sum_{x\in \Sigma_k} T_xPT_x^*
\]
is a projection. Then as in the proof of \cite[Proposition~7.2]{LRR}, the KMS condition implies that
\begin{align*}
\phi(P_n)&=\sum_{k=0}^n\sum_{x\in \Sigma_k}\phi(T_xPT_x^*)
=\sum_{k=0}^n\sum_{x\in \Sigma_k}e^{-\beta k}\phi(PT_x^*T_x)\\
&=\phi(P)\sum_{k=0}^ne^{-\beta k}d^k=(1-e^{-\beta} d)\sum_{k=0}^ne^{-\beta k}d^k
\end{align*}
converges to $1$ as $n\to \infty$. It follows from \cite[Lemma~7.3]{LRR} that for each $c\in C^*(G,P)$, we have $\phi(P_ncP_n)\to \phi(c)$ as $n\to \infty$.
We now use the KMS condition to simplify
\begin{align*}
\phi(c)=\lim_{n\to\infty}\phi(P_ncP_n)&=\lim_{n\to\infty}\sum_{j,k=0}^n\;\sum_{x\in \Sigma_j,\;y\in \Sigma_k}\phi\big((T_xPT_x^*)c(T_yPT_y^*)\big))\\
&=\lim_{n\to\infty}\sum_{j,k=0}^n\;\sum_{x\in \Sigma_j,\;y\in \Sigma_k}e^{-\beta j}\phi\big(PT_x^*cT_yPT_y^*T_xP\big))\\
&=\lim_{n\to\infty}\sum_{k=0}^n\sum_{x\in \Sigma_k}e^{-\beta k}\phi(PT_x^*cT_xP)\\
&=\lim_{n\to\infty}(1-e^{-\beta}d)\sum_{k=0}^n\sum_{x\in \Sigma_k}e^{-\beta k}\phi_P(T_x^*cT_x).
\end{align*}
This is an analogue of the reconstruction formula of \cite[Proposition~7.2]{LRR}.

To finish off, we take $c=T_{b^t}$ in the reconstruction formula. Then for $x\in \Sigma_k$, 
$T_x^*T_{b^t}T_x$ has the form $T_x^*T_{\stem(b^tx)}T_{b^s}$, which vanishes unless $\stem(b^tx)=x$. In that case Lemma~\ref{idsummands} implies that $d$ divides $c^jd^{-j}t$ for all $j<k$, $b^tx =xb^{c^kd^{-k}t}$ and $T_x^*T_{\stem(b^tx)}T_{b^{c^kd^{-k}t}}=T_{b^{c^kd^{-k}t}}$. (We have to worry separately about $k=0$, though.)
Thus  we have
\[
\phi(T_{b^t})=(1-e^{-\beta}d)\Big(\phi_P(T_{b^t})+\sum_{\{k\geq 1\,:d\,\mid\, c^jd^{-j}t\text{ for }0\leq j<k\}}e^{-\beta k}\phi_P(T_{b^{c^kd^{-k}t}})\Big),
\]
which by  definition of the measure $\mu$ is the same as $\psi_{\beta,\mu}(T_{b^t})$. Thus Corollary~\ref{btsuffices} implies that $\phi=\phi_{\beta,\mu}$.

Now we add the assumption that $d$ does not divide $c$, and aim to prove that $\mu\mapsto\psi_{\beta,\mu}$ is injective.
So we suppose that $\mu$ and $\nu$ are probability measures on $\T$ satisfying $\psi_{\beta,\mu}=\psi_{\beta,\nu}$. To simplify the notation, we observe that the integrals appearing in our formulas are the moments
\[
M_n(\mu):=\int_{\T} z^n\,d\mu
\]
of the measures. Since the non-negative moments characterise a probability measure, we will prove that $M_t(\mu)=M_t(\nu)$ for all $t\in\N$.

We prove by induction on $k$ that, if $t\in \N$, $d$ divides $c^jd^{-j}t$ for all $j<k$ and $d$ does not divide $c^kd^{-k}t$, then
\begin{equation}\label{indhyp}
M_{c^jd^{-j}t}(\mu)=M_{c^jd^{-j}t}(\nu)\quad\text{for all $j$ satisfying $0\leq j\leq k$.}
\end{equation} 
For $k=0$, which we interpret to mean that $d$ does not divide $t$, the sum in \eqref{psivsmoments} is absent, and we have
\[
(1-e^{-\beta}d)M_t(\mu)=\psi_{\beta,\mu}(T_{b^t})=\psi_{\beta,\nu}(T_{b^t})=(1-e^{-\beta}d)M_t(\nu).
\]
Since $e^{\beta}>d$, we deduce that $M_t(\mu)=M_t(\nu)$.

Suppose that the inductive hypothesis is true for $k$, and we have $s\in \N$ such that $d$ divides $c^jd^{-j}s$ for $j<k+1$ and $d$ does not divide $c^{k+1}d^{-(k+1)}s$. Then \eqref{psivsmoments} gives
\[
\psi_{\beta,\mu}(T_{b^s})=(1-e^{-\beta} d)\big(M_s(\mu)+e^{-\beta}dM_{cd^{-1}s}(\mu)+\cdots +e^{-\beta (k+1)}d^{k+1}M_{c^{k+1}d^{-(k+1)}s}(\mu)\big),
\]
and hence $\psi_{\beta,\mu}=\psi_{\beta,\nu}$ implies that
\begin{align}\label{eqmoments}
M_s(\mu)+e^{-\beta}d&M_{cd^{-1}s}(\mu)+\cdots +e^{-\beta (k+1)}d^{k+1}M_{c^{k+1}d^{-(k+1)}s}(\mu)\\
&=M_s(\nu)+e^{-\beta}dM_{cd^{-1}s}(\nu)+\cdots +e^{-\beta (k+1)}d^{k+1}M_{c^{k+1}d^{-(k+1)}s}(\nu).\notag
\end{align}
Notice that $t=cd^{-1}s$ has the property that $d$ divides $c^jd^{-j}t$ for all $j<k$ and $d$ does not divide $c^kd^{-k}t$. Thus the inductive hypothesis implies that
\[
M_{c^{j+1}d^{-(j+1)}s}(\mu)=M_{c^jd^{-j}t}(\mu)=M_{c^jd^{-j}t}(\nu)=M_{c^{j+1}d^{-(j+1)}s}(\nu) \text{ for $0\leq j<k$.}
\]
But this says precisely that \eqref{indhyp} holds for $1\leq j\leq k+1$. Now cancelling terms in \eqref{eqmoments} gives $M_s(\mu)=M_s(\nu)$, which is \eqref{indhyp} for the remaining case $j=0$.  
 Thus we have proved the inductive hypothesis for $k+1$, and this completes the proof by induction.

Now we fix $t\in \N$. Since $d$ does not divide $c$, there is a first $k$ such that $d$ does not divide $c^kd^{-k}t$, and then we have \eqref{indhyp} for this $k$. But now taking $j=0$ in \eqref{indhyp} shows that $M_t(\mu)=M_t(\nu)$. Thus $\mu$ and $\nu$ have the same moments, and $\mu=\nu$. Thus $\mu\mapsto \psi_{\beta,\mu}$ is injective.

Since the sum in \eqref{psivsmoments} is always finite, the assignment $\mu\mapsto \psi_{\beta,\mu}$ is affine and continuous for the weak* topologies. We have shown that it is a bijection of the compact space $P(\T)$ onto the simplex of KMS$_\beta$ states. Hence it is a homeomorphism, and this completes the proof of Theorem~\ref{KMSToe}.

\section{KMS states at the critical inverse temperature}

Recall from Corollary~\ref{restonbeta} that every KMS$_{\ln d}$ state of $C^*(G,P)$ factors through the quotient map of $C^*(G,P)$ onto the Cuntz algebra $\OO(G,P)$. We write $\bar T_x$ for the image of $T_x\in C^*(G,P)$ in $\OO(G,P)$.

\begin{prop}\label{KMScrit}
There is a KMS$_{\ln d}$ state $\psi$ on $(\OO(G,P),\alpha)$ such that
\[
\psi(\bar T_x\bar T_y^*)=\delta_{x,y}e^{-\beta\theta(x)}.
\]
If $d$ does not divide $c$, then this is the only KMS state on $(\OO(G,P),\alpha)$. 
\end{prop}

\begin{lemma}\label{offdiag0}
Suppose that $d$ does not divide $c$ and $\phi$ is a KMS$_{\ln d}$ state of the Cuntz system $(\OO(G,P),\alpha)$. Then $\phi(\bar T_{b^t})=0$ for all $t\not=0$.
\end{lemma}

\begin{proof}
Suppose that $\phi(\bar T_{b^t})\not=0$.
Since $\{S_j:=\bar T_{b^ja}:0\leq j<d\}$ is a Cuntz family, so is 
\[
\{S_\mu:=S_{\mu_1}S_{\mu_2}\cdots S_{\mu_k}: 0\leq\mu_i<d\}=\{\bar T_{\sigma}:\sigma\in \Sigma_k\}.
\]
Thus for every $k$, 
\[
\phi(\bar T_{b^t})=\phi\Big(\bar T_{b^t}\sum_{\sigma\in \Sigma_k}\bar T_{\sigma}\bar T_{\sigma}^*\Big)=\sum_{\sigma\in \Sigma_k}\phi(\bar T_{b^t\sigma}\bar T_{\sigma}^*)
\]
and there exists $\sigma\in \Sigma_k$ such that $\phi(\bar T_{b^t\sigma}\bar T_{\sigma}^*)\not=0$. Proposition~\ref{charKMS} then implies that $b^t\sigma\vee \sigma<\infty$, and since $b^t\sigma$ and $\sigma$ have the same height, we must have $\stem(b^t\sigma)=\sigma$. Thus we can apply Lemma~\ref{idsummands} for every $k$, and deduce that $d$ divides $c^jd^{-j}t$ for all $j\geq 1$. Write $g:=\gcd(c,d)$, and define $c_1:=g^{-1}c$ and $d_1:=g^{-1}d$. Then $\gcd(c_1,d_1)=1$ and for all $j$ we have
\begin{align*}
d\text{ divides }c^jd^{-j}t&\Longrightarrow d_1\text{ divides }c_1^jd_1^{-j}t\\
&\Longrightarrow d_1\text{ divides }d_1^{-j}t\\
&\Longrightarrow d_1^{j+1}\text{ divides }t.
\end{align*} 
Since $d$ does not divide $c$, we have $d_1>1$, and hence $t=0$.
\end{proof}

\begin{proof}[Proof of Proposition~\ref{KMScrit}]
We choose a decreasing sequence $\{\beta_n\}$ such that $\beta_n\to \ln d$, and take $\mu$ to be the Haar measure on $\T$. Then by passing to a subsequence, we may assume that $\{\psi_{\beta_n,\mu}\}$ converges weak* to a state $\phi$ of $(C^*(G,P),\alpha)$. Then it follows from \cite[Proposition~5.3.23]{BR} that $\phi$ is a KMS$_{\ln d}$ state of $(C^*(G,P),\alpha)$. Corollary~\ref{restonbeta}  implies that $\phi$ factors through a state $\psi$ of $(\OO(G,P),\alpha)$. Since the non-zero moments of the Haar measure all vanish, we have $\psi(\bar T_{b^t})=0$ for all $t\not=0$, and then the formula for $\psi$ follows.

For uniqueness, we use Lemma~\ref{offdiag0} to see that any other KMS$_{\ln d}$ state agrees with $\psi$ on the elements $\bar T_{b^t}$, and hence by Corollary~\ref{btsuffices} on all of $\OO(G,P)$.
\end{proof}

\begin{ex}\label{ddividesc}
Suppose that $d$ divides $c$, that $\beta>\ln d$, and that $\mu\in P(\T)$. Then either $d$ divides $c^jd^{-j}t$ for all $j\geq 0$, or $d$ does not divide $t$. Thus
\[
\psi_{\beta,\mu}(T_{b^t})=\begin{cases}(1-e^{-\beta}d)\big(M_t(\mu)+\sum_{k=1}^\infty e^{-\beta k}d^kM_{c^kd^{-k}t}(\mu)\big)&\text{if $d$ divides $t$}\\
(1-e^{-\beta}d)M_t(\mu)&\text{if $d$ does not divide $t$.}
\end{cases}
\]

If $d=c$, then the only moment appearing in our formula is $M_t(\mu)$, and summing the geometric series shows that
\[
\psi_{\beta,\mu}(T_{b^t})=\begin{cases}M_t(\mu)&\text{if $d$ divides $t$}\\
(1-e^{-\beta}d)M_t(\mu)&\text{otherwise.}
\end{cases}
\]
Now the procedure in the proof of Proposition~\ref{KMScrit} gives a KMS$_{\ln d}$ state on $(\OO(G,P),\alpha)$, and measures with different moments $M_{dn}$ will give different states.

If $d\not=c$, then we write $\delta_z$ for the point mass at $z\in\T$, and take 
\[
\mu=\frac{1}{cd^{-1}}\sum_{w^{cd^{-1}}=1}\delta_w.
\]
Then $\mu$ is a probability measure with moments
\[
M_n(\mu)=\begin{cases}
1&\text{if $cd^{-1}$ divides $n$}\\
0&\text{otherwise.}
\end{cases}
\]
Thus
\[
\psi_{\beta,\mu}(T_{b^t})=\begin{cases}1&\text{if $cd^{-1}$ and $d$ divide $t$}\\
(1-e^{-\beta}d)&\text{if $cd^{-1}$ divides $t$ and $d$ does not}\\
0&\text{if $cd^{-1}$ does not divide $t$.}
\end{cases}
\]
If we now choose $\beta_n$ decreasing to $\ln d$ as in the proof of Proposition~\ref{KMScrit}, we get a KMS$_{\ln d}$ state $\psi_\mu$ on $(\OO(G,P),\alpha)$ such that
\[
\psi_{\mu}(\bar T_{b^t})=\begin{cases}1&\text{if $cd^{-1}$ and $d$ divide $t$}\\
0&\text{otherwise}.\end{cases}
\]

So uniqueness of the KMS$_{\ln d}$ state fails whenever $d$ divides $c$.
\end{ex}

\begin{ex}
The case $d=c$ is quite different. Then the unitary element $\bar T_{b^c}$ commutes with everything. It and the Cuntz family $\{\bar T_{b^ja}:0\leq j<c\}$ generate $\OO(G,P)$: indeed, $\bar T_a$ is a member of the Cuntz family, and we can recover $\bar T_b$ using the formula
\[ 
\sum_{j=0}^{c-2}\bar T_{b^{j+1}a}\bar T_{b^ja}^*+ \bar T_{b^c}\bar T_a\bar T_{b^{c-1}a}=\bar T_b\Big(\sum_{j=0}^{c-1}\bar T_{b^ja}\bar T_{b^ja}^*\Big)=\bar T_b.
\]
The Cuntz family gives us a homomorphism $\pi:\OO_c=C^*(S_j)\to \OO(G,P)$ such that $\pi(S_j)=\bar T_{b^ja}$, and the unitary $\bar T_{b^c}$ gives us a homomorphism $\rho:C(\T)\to \OO(G,P)$ such that $\rho(\iota)=\bar T_{b^c}$. Since they have commuting ranges, they induce a homomorphism $\pi\otimes \rho$ of $\OO_c\otimes C(\T)$ onto $\OO(G,P)$. This is in fact an isomorphism because the universal property of $\OO(G,P)$ gives an inverse.
\end{ex}

\begin{rmk}\label{sortleftright}
In the definition of the Baumslag-Solitar group $\BS(c,d)$, the integers $c$ and $d$ play an equal role. So at first sight it seems strange that the critical inverse temperature is dictated by $d$ alone. However, right at the beginning of his theory, Nica made a critical choice: his covariance relation was modelled on the behaviour of the \emph{left}-regular representation. If he had started with the right-regular representation, his theory would have looked quite different.

The right-regular representation $\rho$ of a group $G$ is characterised in terms of the usual basis $\{e_g:g\in G\}$ for $\ell^2(G)$ by $\rho_he_g=e_{gh^{-1}}$ (the inverse has to be there to ensure that $\rho_g\rho_h=\rho_{gh}$). So the natural representation of $P$ by operators $\{R_x:x\in P\}$ on $\ell^2(P)$ is given by
\[
R_xe_y=\begin{cases}0&\text{if $y\notin Px$}\\
e_{yx^{-1}}&\text{if $y\in Px$.}
\end{cases}
\]
Each $R_x$ is a \emph{coisometry}: $R_x^*$ is an isometry. In other words, $R_x$ is a partial isometry with initial projection $R_x^*R_x:\ell^2(P)\to \ell^2(Px)$ and range projection $R_xR_x^*=1$.

The analogue of Nica's partial order is the right-invariant order defined by $x\leq_{\rt} y\Longleftrightarrow y\in Px$ (and for the duration of this remark, we'll write $\leq_{\lt}$ for the usual left-invariant one). There is an analogous notion of ``right-quasi-lattice ordered'' involving least upper bounds $x\vee_{\rt}y$ with respect to $\leq_{\rt}$, and the analogue of Nica covariance is the relation
\begin{equation}\label{rtNica}
(R_x^*R_x)(R_y^*R_y)=\begin{cases}R_{x\vee_{\rt} y}^*R_{x\vee_{\rt} y}&\text{if $x\vee_{\rt}y<\infty$}\\
0&\text{if $x\vee_{\rt}y=\infty$,}
\end{cases}
\end{equation}
which is satisfied by the right-regular representation. One can then get a universal $C^*$-algebra, which we will denote by $C^*(G,P,\leq_{\rt})$.

Fortunately, there is a device for studying this $C^*$-algebra (for which we thank Ilija Tolich). Consider the opposite group $G^{\op}=\{g^\flat:g\in G\}$ with $g^\flat h^\flat=(hg)^\flat$, and the corresponding subsemigroup $P^\op$. Then the usual partial order $\leq_{\lt}$ on $G^\op$ satisfies
\[
g^\flat\leq_{\lt} h^\flat\Longleftrightarrow h^\flat\in g^\flat P^\op=(Pg)^\flat\Longleftrightarrow h\in Pg\Longleftrightarrow g\leq_{\rt} h;
\]
we deduce that $(G^\op,P^\op)$ is quasi-lattice ordered in the usual sense if and only if $(G,P,\leq_\rt)$ is right-quasi-lattice ordered, with $g^\flat\vee_{\lt} h^\flat =(g\vee_{\rt} h)^\flat$. One can check quite easily that $R^T:x\mapsto T_{x^\flat}^*$ is a covariant coisometric representation of $P$ in the sense that \eqref{rtNica} holds if and only if $T$ is a Nica-covariant representation of $P^{\op}$. Thus $(C^*(G^\op,P^\op), R^T)$ is universal for coisometric representations satisfying \eqref{rtNica}.

When $G=\BS(c,d)$, the opposite group is 
\[
G^\op=\langle a, b : b^ca = ab^d \rangle=\BS(d,c).
\]
Thus, had we chosen to work with the partial order $\leq_{\rt}$, we would have found a system with a phase transition at inverse temperature $\ln c$. (There is a minor wrinkle: because passing from the isometric representation $T$ to the coisometric representation $R^T$ involves an adjoint, the dynamics satisfies $\alpha^\rt_t(R_x)=e^{-i\theta(x)}R_x$. However, one can argue that this is the natural one, because both this and the usual dynamics are implemented spatially on $\ell^2(P)$ by the unitary representation $U$ such that $U_te_x=e^{it\theta(x)}e_x$.)
\end{rmk}

\section{Ground states}

A state $\phi$ is   a \emph{ground state} of $(C^*(G,P),\alpha)$ if, for all analytic elements $a$ and $b$,  the entire function $z\mapsto \phi(a\alpha_z(b))$ is bounded in the upper half-plane $\im z>0$. The KMS$_\infty$ states are the weak* limits of sequences of KMS$_{\beta_n}$ states as $\beta_n\to \infty$. Every KMS$_\infty$ state is a ground state, but a ground state need not be a KMS$_\infty$ state by \cite[Proposition~5.3.23]{BR} and \cite[Proposition~3.8]{CM}.

\begin{thm}\label{thmground}
Suppose that $\omega$ is a state of the Toeplitz algebra $\TT(\N)=C^*(S)$. Then there is a ground state $\psi_\omega$ of $(C^*(G,P),\alpha)$ such that
\begin{equation}\label{defground}
\psi_\omega(T_xT_y^*)=
\begin{cases}
0&\text{if $\theta(x)\not=0$ or $\theta(y)\not=0$}\\
\omega(S^sS^{*t})&\text{if $x=b^s$ and $y=b^t$.}
\end{cases}
\end{equation}
The state $\psi_\omega$ is a KMS$_\infty$ state if and only if $\omega$ factors through the quotient map $q:\TT(\N)\to C(\T)$. The map $\omega\mapsto \psi_\omega$ is an affine isomorphism of the state space of $\TT(\N)$ onto the compact convex set of ground states.
\end{thm}

There are many states of $\TT(\N)$ which do not factor through $q:\TT(\N)\to C(\T)$: for example, the vector states given by unit vectors in $\ell^2$. Thus Theorem~\ref{thmground} implies that the system $(C^*(G,P),\alpha)$ has many ground states which are not KMS$_\infty$ states. Thus (in the terminology of \cite{CM}) the system admits a second phase transition at $\beta=\infty$.

We now do some preparation for the the proof of Theorem~\ref{thmground}. First we need to be able to recognise ground states.

\begin{lemma}\label{idground}
A state $\psi$ of $(C^*(G,P),\alpha)$ is a ground state if and only if
\begin{equation}\label{charground}
\psi(T_xT_y^*)\not=0\Longrightarrow \theta(x)=\theta(y)=0.
\end{equation}
\end{lemma}

\begin{proof}
%We begin by observing that 
Suppose that $\psi$ is a ground state and $\psi(T_xT_y^*)\not=0$. Then
\[
\big|\psi(T_x\alpha_{r+is}(T_y^*))\big|=\big| e^{-i(r+is)\theta(y)}\psi(T_xT_y^*)\big|=e^{s\theta(y)}|\psi(T_xT_y^*)|
\]
is bounded on the upper half-plane $s\geq 0$, and hence $\theta(y)=0$.  Since $\psi(T_yT_x^*)=\overline{\psi(T_xT_y^*)}\not=0$, we also deduce that $\theta(x)=0$.

Next suppose that $\psi$ is a state satisfying \eqref{charground}. Let $X=T_xT_y^*$ and $Y$ be analytic elements for $\alpha$. Then the Cauchy-Schwarz inequality gives
\begin{align}\label{CSground}
|\psi(Y\alpha_{r+is}(X))|^2&=e^{-2s(\theta(x)-\theta(y))}|\psi(YT_xT_y^*)|^2\\
&\leq e^{-2s(\theta(x)-\theta(y))}\psi(Y^*Y)\psi(T_yT_x^*T_xT_y^*)\notag\\
&= e^{-2s(\theta(x)-\theta(y))}\psi(Y^*Y)\psi(T_yT_y^*).\notag
\end{align}
If $\theta(y)\not= 0$ then \eqref{charground} implies that $\psi(T_yT_y^*)=0$, and the right-hand side of \eqref{CSground} is trivially bounded. If $\theta(y)=0$, then the right-hand side of \eqref{CSground} is bounded by $\psi(Y^*Y)\psi(T_yT_y^*)$. So either way, $|\psi(Y\alpha_{r+is}(X))|$ is bounded for $s\geq 0$, and $\psi$ is a ground state.
\end{proof}

Next we need a good supply of representations. Our basic construction was inspired by our earlier one using induced representations. 

We continue to use the orthonormal basis $\{e_{k,\sigma}:\sigma\in \Sigma_k\}$ for $\ell^2(\Sigma_k)$.

\begin{lemma}\label{defUV}
Suppose that $W$ is an isometry of a Hilbert space $H$. Then there are isometries $U$ and $V$ on $\bigoplus_{k=0}\ell^2(\Sigma_k)\otimes H$ such that
\begin{align*}
U(e_{k,\sigma}\otimes h)&=e_{k,\stem(b\sigma)}\otimes W^sh\quad\text{where $b\sigma=\stem(b\sigma)b^s$, and }\\
V(e_{k,\sigma}\otimes h)&=e_{k+1,a\sigma}\otimes h.
\end{align*}
\end{lemma}

\begin{proof}
Lemma~\ref{maponstems} implies that $\sigma\mapsto \stem(b\sigma)$ is a bijection of $\Sigma_k$ onto $\Sigma_k$. Thus if $\{h_i:i\in I\}$ is an orthonormal basis for $H$, then $\{e_{k,\stem(b\sigma)}\otimes W^sh_i:\sigma\in \Sigma,\,i\in I\}$ is an orthonormal set in $\ell^2(\Sigma_k)\otimes H$ for each $k$. Thus there is an isometry $U$ as claimed. Since each $a\sigma$ is already a stem, Lemma~\ref{maponstems} also implies that $\sigma\mapsto a\sigma$ is an injection of $\Sigma_k$ in $\Sigma_{k+1}$ for each $k$. Thus $\{e_{k+1,a\sigma}\otimes h_i\}$ is also orthonormal, and there is an isometry $V$ with the required property.
\end{proof}

\begin{prop}
Suppose that $W$ is an isometry of a Hilbert space $H$, and $U$, $V$ are as in Lemma~\ref{defUV}. Then $U$ and $V$ satisfy the relations \eqref{t1}, \eqref{t4} and \eqref{t5} of Proposition~\ref{defrel}.
\end{prop}

\begin{proof}
The calculations in the fourth paragraph of the proof of Proposition~\ref{indrepconst} show that $U$ and $V$ satisfy \eqref{t1} and \eqref{t5}. To verify \eqref{t4}, we need a formula for $U^*$. We claim that 
\begin{equation}\label{formU*}
U^*(e_{k,\tau}\otimes h)=e_{k,\rho}\otimes W^{*t}h \quad\text{where $\rho\in \Sigma_k$ satisfies $\tau b^t=b\rho$;}
\end{equation}
Lemma~\ref{maponstems} implies that there is a unique stem $\rho$ such that $b\rho$ begins with $\tau$, and then $t$ is uniquely determined by $\tau b^t=b\rho$. To prove the claim, we compare
\begin{equation}\label{lhs}
\big(e_{k,\rho}\otimes W^{*t}h\,|\,e_{k,\sigma}\otimes g\big)=\delta_{\rho,\sigma}(h\,|\,W^tg)\quad\text{where $\tau b^t=b\rho$}
\end{equation}
with 
\begin{equation}\label{rhs}
\big(e_{k,\tau}\otimes h\,|\,U(e_{k,\sigma}\otimes g)\big)
=\delta_{\tau,\stem(b\sigma)}(h\,|\,W^sg)\quad\text{where $b\sigma=\stem(b\sigma)b^s$.}
\end{equation}
First, suppose that $\rho=\sigma$. Then $b\sigma=\stem(b\sigma)b^s=\stem(b\rho)b^s=\stem(\tau b^t)b^s=\tau b^s$ because $\tau$ is a stem. Thus $\tau=\stem(b\sigma)$. Now $\tau b^t=b\rho=b\sigma=\stem(b\sigma)b^s=\tau b^s$, and hence $s=t$. Thus \eqref{lhs} and \eqref{rhs} agree.
Second, suppose that $\rho\neq \sigma$. By Lemma~\ref{maponstems}
\eqref{hkm}, $x\mapsto \stem(bx)$ is a bijection on $\Sigma_k$, and hence $\stem(b\sigma)\neq\stem(b\rho)=\stem(\tau b^t)=\tau$, and both \eqref{lhs} and \eqref{rhs} are $0$. This proves the claim.

We now compute the right-hand side of \eqref{t4}:
\begin{align*}
U^{d-1}VU^{*c}(e_{k,\sigma}\otimes h)&=U^{d-1}V(e_{k,\mu}\otimes W^{*t}h)\quad\text{where $\mu\in \Sigma_k$ satisfies $\sigma b^t=b^c\mu$}\\
&=e_{k+1,\stem(b^{d-1}a\mu)}\otimes W^sW^{*t}h\quad\text{where $b^{d-1}a\mu=\stem(b^{d-1}a\mu)b^s$}\\
&=e_{k+1,b^{d-1}a\mu}\otimes W^{*t}h
\end{align*}
because $b^{d-1}a\mu$ is a stem. The left-hand side of \eqref{t4} is
\[
U^*V(e_{k,\sigma}\otimes h)=e_{k+1,\rho}\otimes W^{*r}h \quad\text{where $\rho\in \Sigma_{k+1}$ satisfies $a\sigma b^r=b\rho$.} 
\]
Now the equation
\[
b(b^{d-1}a\mu)=b^da\mu=ab^c\mu=a\sigma b^t
\]
implies that $\rho=b^{d-1}a\mu$ (because $\rho$ is the unique stem such that $b\rho$ begins with $a\sigma$) and then $r=t$. Thus \eqref{t4} follows.
\end{proof}

\begin{cor}\label{absground}
Suppose that $W$ is an isometry on a Hilbert space $H$, and $U$, $V$ are the isometries described in Lemma~\ref{defUV}. Let $\pi_{U,V}$ be the corresponding representation of $C^*(G,P)$ on $\bigoplus_{k\geq 0}\ell^2(\Sigma_k)\otimes H$. Then for every unit vector $h$ in $H$, there is a ground state  $\psi_{h,W}$ of $(C^*(G,P),\alpha)$ such that
\[
\psi_{h,W}(a)=\big(\pi_{U,V}(a)(e_{0,e}\otimes h)\,|\,e_{0,e}\otimes h\big).
\]
\end{cor}

\begin{proof}
Since $T_x$ maps $\ell^2(\Sigma_0)\otimes H$ into $\ell^2(\Sigma_{\theta(x)})\otimes H$, 
\[
\psi_{h,W}(T_xT_y^*)=\big(\pi_{U,V}(T_y)^*(e_{0,e}\otimes h)\,|\,\pi_{U,V}(T_x)^*(e_{0,e}\otimes h)\big)
\]
vanishes unless $\theta(x)=0=\theta(y)$. So Lemma~\ref{idground} implies that $\psi_{h,W}$ is a ground state.
\end{proof}

\begin{proof}[Proof of Theorem~\ref{thmground}]
Suppose $\omega$ is a state of $\TT(\N)=C^*(S)$. We consider the GNS representation $\pi_\omega$ of $\TT(\N)$ on $H_\omega$ with cyclic vector $\xi_\omega$, from which we can recover $\omega$ via the formula $\omega(c)=(\pi_\omega(c)\xi_\omega\,|\,\xi_\omega)$. Applying Corollary~\ref{absground} with $W=\pi_\omega(S)$, $U$ and $V$ the isometries of Lemma~\ref{defUV},  and $h=\xi_\omega$ gives a ground state $\psi_\omega:=\psi_{\xi_\omega, \pi_\omega(S)}$ of $(C^*(G,P),\alpha)$ such that 
\[
\psi_\omega(a)=\big(\pi_{U,V}(a)(e_{0,e}\otimes \xi_\omega)\,|\,e_{0,e}\otimes \xi_\omega\big).
\]
We need to verify the formula \eqref{defground}.

Since $V=\pi_{U,V}(T_a)$ maps $\ell^2(\Sigma_0)\otimes H_\omega=\C e_{0,e}\otimes H_\omega$ into $\ell^2(\Sigma_1)\otimes H_\omega$, we have
\[
\psi_\omega(T_xT_y^*)=0 \text{\ \ unless $\theta(x)=0=\theta(y)$.}
\]
If $\theta(x)=0=\theta(y)$, then $x=b^s$ and $y=b^t$ for some $s,t\in \N$, and
\begin{align*}
\psi_\omega(T_xT_y^*)&=\psi_\omega(T_b^sT_b^{*t})\\
&=\big(\pi_{U,V}(T_b^sT_b^{*t})(e_{0,e}\otimes \xi_\omega)\,|\,e_{0,e}\otimes \xi_\omega\big)\\
&=\big(U^sU^{*t}(e_{0,e}\otimes \xi_\omega)\,|\,e_{0,e}\otimes \xi_\omega\big)\\
&=\big(U^{*t}(e_{0,e}\otimes \xi_\omega)\,|\,U^{*s}(e_{0,e}\otimes \xi_\omega)\big).
\end{align*}
Since $\Sigma_0=\{e\}$, the formula \eqref{formU*} for $U^*$ collapses to $U^*(e_{0,e}\otimes h)=e_{0,e}\otimes W^*h$, and we have
\begin{align*}
\psi_\omega(T_xT_y^*)&=\big(e_{0,e}\otimes \pi_{\omega}(S)^{*t}\xi_\omega\,|\,e_{0,e}\otimes \pi_{\omega}(S)^{*s}\xi_\omega\big)\\
&=\big(\pi_{\omega}(S)^{*t}\xi_\omega\,|\,\pi_{\omega}(S)^{*s}\xi_\omega\big)\\
&=\big(\pi_{\omega}(S^sS^{*t})\xi_\omega\,|\,\xi_\omega\big)\\
&=\omega(S^sS^{*t}),
\end{align*}
as in \eqref{defground}.

Next we suppose that $\psi_\omega$ is a KMS$_\infty$ state. Then there are an increasing sequence $\beta_n\to \infty$ and KMS$_{\beta_n}$ states $\phi_n$ such that $\phi_n$ converges weak* to $\psi_\omega$. Corollary~\ref{restonbeta} implies that each $\phi_n$ factors through the quotient by the ideal generated by $1-T_bT_b^*$, and hence so does the limit $\psi_\omega$. The kernel of $q$ is spanned by the elements $S^m(1-SS^*)S^{*n}$ (they are a family of matrix units spanning $\ker q=\KK(\ell^2)$), and the formula \eqref{defground} implies that
\[
\omega(S^m(1-SS^*)S^{*n})=\psi_\omega(T_{b^m}(1-T_bT_b^*)T_{b^n}^*)=0.
\]
Thus $\omega$ factors through $q$.

Conversely, suppose that $\omega$ factors through $q$. Then there is a probability measure $\mu$ on $\T$ such that $\omega(c)=\int q(c)\,d\mu$ for all $c\in \TT(\N)$. Choose a sequence $\beta_n$ with $\beta_n\to \infty$. Then for each $n$, the state $\psi_{\beta_n,\mu}$ is determined by Corollary~\ref{btsuffices} and the formula \eqref{psivsmoments} for $\psi_{\beta_n,\mu}(T_{b^t})$ in Proposition~\ref{KMSToe}. The sum on the right-hand side of \eqref{psivsmoments} is finite, and for each $k$ we have
\[
e^{-\beta_nk}d^k\int_{\T} z^{c^kd^{-k}t}\,d\mu(z)\to 0\quad\text{as $n\to \infty$.}
\]
Since we also have $1-e^{-\beta_n}d\to 1$, we deduce that
\[
\psi_{\beta_n,\mu}(T_b^t)\to \int z^t\,d\mu(z)=\int q(S^t)\,d\mu=\psi_\omega(T_{b^t}).
\]
Thus $\psi_\omega$ is a KMS$_\infty$ state.

The formula \eqref{defground} shows that $\omega\mapsto \psi_\omega$ is affine, weak* continuous and one-to-one. To see that is is onto, suppose $\phi$ is a ground state. Since $T_b$ is a non-unitary isometry, Coburn's theorem implies that there is an isomorphism $\pi_{T_b}$ of $\TT(\N)$ into $C^*G,P)$ such that $\pi_{T_b}(S)=T^b$, and then $\omega:=\psi\circ\pi_{T_b}$ is a state of $\TT(\N)$. Lemma~\ref{idground} implies that $\phi$ vanishes on all spanning elements except those of the form $T_{b^s}T_{b^t}^*$, and formula \eqref{defground} shows that $\phi$ agrees with $\psi_{\omega}$ on all spanning elements. Thus $\phi=\psi_\omega$, and $\omega\mapsto \psi_\omega$ is onto. Now we can deduce that it is a homeomorphism of the compact state space of $\TT(\N)$ onto the compact set of ground states.
\end{proof}

\appendix\label{app}

\section{Amenability of $(G,P)$}

In setting up our conventions, we implicitly assumed that $(G,P)$ is amenable in the sense of Nica, and here we prove this. This result is not a surprise, since Spielberg proved that his groupoid model is amenable \cite[Theorem~3.23]{SBS}, and the various notions of amenability are all meant to do the same thing. Nevertheless, it is fairly routine to see it directly. For the purposes of this appendix, it is helpful to distinguish between the Toeplitz representation $T$ of $P$ on $\ell^2(P)$ and the universal representation of $P$ in $C^*(G,P)$, which we denote by $i$ (following \cite{LRold}). 

\begin{thm}
\label{BSamenable}
 %Suppose $c$ and $d$ are positive integers, $G := \langle a, b : ab^c = b^da \rangle$ is the Baumslag-Solitar group
%and $P$ is the submonoid  of $G$ generated
%by $a$ and $b$.  
The quasi-lattice ordered group $(G,P)$ is amenable.
\end{thm}

We follow the argument of \cite[Proposition~4.2]{LRold}, using the height map $\theta:G\to \Z$ in the role of the map $\phi$ in that proposition (which was later described as a ``controlled map'' in \cite[\S4]{CL2}). Unfortunately, that proposition does not apply as it stands, since $\theta$ does not have the property (i) required of controlled maps in the statement of  \cite[Proposition~4.2]{LRold} --- for example, we have $a\leq ab$, but $\theta(a)=\theta(ab)$. But the general idea works.

By \cite[Corollary~3.3]{LRold}, there is a contraction $\Phi:C^*(G,P) \to \clsp\{i(x)i(x)^*:x\in P\}$ such that \[
\Phi(i(x)i(y)^*) = \begin{cases} i(x)i(x)^* & \text{ if $x=y$}\\0& \text{ otherwise.}\end{cases}
\]
By  \cite[Definition~3.4]{LRold},  $(G,P)$ is amenable if $\Phi$ is faithful in the sense that
$\Phi(R^*R)=0$ implies $R=0$. We consider the dual action $\hat\theta:\T\to \Aut C^*(G,P)$ characterised by $\hat\theta_z(i(x))=z^{\theta(x)}i(x)$. Our strategy is to  analyse the structure of the fixed-point algebra $C^*(G,P)^\theta=\clsp\{i(x)i(y)^*:\theta(x)=\theta(y)\}$ for this action, and show that $\Phi$ factors through the conditional expectation  $\Phi^\theta$ of $C^*(G,P)$ onto   $C^*(G,P)^\theta$ obtained by averaging over $\T$. 

\begin{lemma}
\label{lem:apdefBk}
For $k \geq 0$, the algebraic linear span
\[B_k := \lsp \{i(\sigma)D  i(\tau)^*:\sigma, \tau \in \Sigma_k \text{ and } D \in C^*(i(b))\}\]
is a closed $C^*$-subalgebra of $C^*(G,P)^{\theta}$.
\end{lemma}

\begin{proof}
Since $\{i(\sigma):\sigma \in \Sigma_k\}$ is a Toeplitz-Cuntz family, $\{i(\sigma)i(\tau)^* : \sigma, \tau \in \Sigma_k\}$ is 
a set of matrix units in the $C^*$-algebra $\overline{B_k}$.
This gives  a homomorphism $\phi:M_{\Sigma_k}(\C) \to \overline{B_k}$ which maps the usual matrix units $\{E_{\sigma\tau}:\sigma,\tau\in \Sigma_k\}$ to $\{i(\sigma)i(\tau)^*\}$. 
There is also a unital homomorphism
$\psi:C^*(i(b)) \to \overline{B_k}$ such that 
 \[\psi(D) = \sum_{\sigma\in\Sigma_k}i(\sigma)Di(\sigma)^*.\] We have
\begin{align*}
 \phi(E_{\sigma\tau})\psi(D)&= i(\sigma)i(\tau)^*\sum_{\mu}i(\mu)Di(\mu)^*=i(\sigma)i(\tau)^*i(\tau)Di(\tau)^*\\
 &=i(\sigma)Di(\tau)^*=\psi(D)\phi(E_{\sigma\tau}).
\end{align*} 
Each $A \in M_{\Sigma_k}(\C)$ is a linear combination of the $E_{\sigma\tau}$, and hence $\psi(D)\phi(A) = \phi(A)\psi(D)$ for all $A \in M_{\Sigma_k}(\C)$ and $D \in C^*(i(b))$.

Since the ranges of $\phi$ and $\psi$ commute, the universal property of the maximal tensor product gives a homomorphism
$\phi \otimes_{\text{max}}\psi$ of  $M_{\Sigma_k}(\C) \otimes C^*(i(b))$ into $\overline{B_k}$.
We claim that the range of $\phi \otimes_{\text{max}}\psi$ is $B_k$. 
Since  $M_{\Sigma_k}(\C) \otimes C^*(i(b))$ is spanned by elements of the form
$E_{\sigma\tau}\otimes D$ (with no closure, see for example \cite[Theorem~B.18]{tfb}), 
the range of $ \phi \otimes_{\text{max}}\psi$ is spanned by $\phi(E_{\sigma\tau})\psi(D)=i(\sigma)D i(\tau)^*$ (no closure) and hence 
equal to $B_k$.  Thus $B_k$ is a closed $C^*$-subalgebra of $C^*(G,P)^{\theta}$.
\end{proof}

\begin{lemma}
\label{lem:ap1}
 For $k \geq 0$, 
 we have $B_kB_{k+1}=B_{k+1}$.
\end{lemma}

\begin{proof} Since $\{i(\sigma):\sigma\in \Sigma_k\}$ is a Toeplitz-Cuntz family, we have $i(\sigma)^*i(\tau)=0$ unless $\sigma$ extends $\tau$ or vice-versa. So to see $B_kB_{k+1}\subset B_{k+1}$, it suffices to take $\sigma, \tau, \mu, \nu \in \Sigma_k$, $\mu',\nu' \in \Sigma_1$, $C, D\in C^*(i(b))$, and show that
\begin{equation}\label{eq:app.5}i(\sigma)D i(\tau)^* i(\mu\mu')C i(\nu \nu')^* = \delta_{\tau, \mu} i(\sigma) D i(\mu') C i(\nu\nu')^*\end{equation}
is in $B_kB_{k+1}$.
Suppose that $D = i(b)^s i(b)^{*t}$ for some $s,t \in \N$.
If $\mu' = b^ja$ for some integer $j \in [0,d)$, then
\begin{equation}\label{eq:app1}Di(\mu') = i(b)^si(b)^{*t}i(\mu') = i(b)^si(b)^{*t}i(b)^ji(a).\end{equation}
Now if $j <t$, then (\ref{eq:app1}) is equal to \[i(b)^{s+j-t}i(a) = i(\stem(b^{s+j-t}a))i(b)^q\]
for some $q\in \N$.
On the other hand, if $t \geq j$, then (\ref{eq:app1}) is equal to 
\[i(b)^si(b)^{*(t-j)}i(a) = i(b)^s i(b)^{(t-j)(d-1)}i(a)i(b)^{*(t-j)c}\]
and we write $i(b)^{s+(t-j)(d-1)}i(a)$ as $i(\stem(b^{s+(t-j)(d-1)}a)i(b)^r$ for some $r\in \N$.
Either way, (\ref{eq:app.5}) has the form $i(\sigma\stem(b^na))C'i(\mu\mu')$ and is in $B_{k+1}$.  
Thus $B_kB_{k+1}\subset B_{k+1}.$

To see the reverse containment, let $\sigma,\tau\in\Sigma_{k+1}$ and write $\sigma=\sigma'\sigma''$ where $\sigma'\in \Sigma_k$ and $\sigma'' \in \Sigma_{1}$.  For $i(\sigma)Di(\tau)^* \in B_{k+1}$ we have
\[i(\sigma)Di(\tau)^*=i(\sigma'\sigma'')Di(\tau)^*=i(\sigma')i(\sigma')^*i(\sigma')i(\sigma'')Di(\tau)^*=\big (i(\sigma')i(\sigma')^*\big)\big(i(\sigma)Di(\tau)^*\big),\] which is in $B_kB_{k+1}$.
This extends to arbitrary elements of $B_{k+1}$ and hence $B_k \subset B_kB_{k+1}$.
\end{proof}

\begin{cor}
\label{cor:ap2}
 For $k \geq 0$, $C_k := B_0 + \dots + B_k$
is a $C^*$-subalgebra of the core $C^*(G,P)^{\theta}$  and
\begin{equation*}%\label{eq-fpa}
C^*(G,P)^{\theta} = \overline{ \textstyle{\bigcup_{k=0}^{\infty}} C_k}.
\end{equation*}
\end{cor}
 
\begin{proof}
We  prove that $C_k$ is a $C^*$-subalgebra of $C^*(G,P)^{\theta}$ by induction on $k$.  
Notice that $C_0 = B_0 = C^*(i(b))$  is a $C^*$-subalgebra. 

Suppose that   $C_k$ is a $C^*$-subalgebra of $C^*(G,P)^{\theta}$ for $k\geq 0$.  Lemma~\ref{lem:ap1} implies that
\[C_kB_{k+1} = (B_0 + \dots + B_k)B_{k+1} =B_0B_1 \dots B_kB_{k+1} + \dots+ B_kB_{k+1}=B_{k+1}.\]
It follows that  $B_{k+1}$ is an ideal in the $C^*$-algebra $A$ generated by $C_k$ and $B_{k+1}$.  Since $C_k$ is a subalgebra of $A$, \cite[Theorem~3.1.7]{Murphy} 
implies that $C_k + B_{k+1} = C_{k+1}$ is a $C^*$-subalgebra of $A$ and hence of  $C^*(G,P)^{\theta}$.

Since $C^*(G,P)^{\theta}=\clsp\{i(x)i(y)^*:\theta(x)=\theta(y)\}$, and such $i(x)i(y)^*\in B_{\theta(x)}\subset C_{\theta(x)}$, it follows that $\bigcup C_k$ is dense in $C^*(G,P)^{\theta}$, and hence $C^*(G,P)^{\theta} = \overline{ \textstyle{\bigcup_{k=0}^{\infty}} C_k}$.
\end{proof}

The Toeplitz representation $T$ of $P$ on $\ell^2(P)$ is Nica covariant, and hence induces a homomorphism $\pi_T$ of $C^*(G,P)$ onto the Toeplitz algebra $\mathcal{T}(G, P):=C^*(T_x:x\in P)$ such that $\pi_T\circ i=T$.
We write
\[H_k = \clsp\{e_{\sigma b^n} : \sigma \in \Sigma_k, n \in \mathbb{N}\} \subset \ell^2(P),
\  \text{and then} \ \ell^2(P) =\bigoplus_{k\geq 0} H_k.\]

\begin{lemma}
\label{lem:ap3}
 For $k \geq 0$,
\begin{enumerate}
 \item \label{it1:applem3}  $H_k$ is invariant for ${\pi_T}|_{B_k}$ and
\item \label{it2:applem3} ${\pi_T}|_{B_k}$ is faithful on $H_k$.
\end{enumerate}
\end{lemma}

\begin{proof}
For item (\ref{it1:applem3}), take
$i(\sigma)i(b)^li(b)^{*m}i(\tau)^* \in B_k $ and $e_{\mu b^n} \in H_k$ where $m,n \in \N$.
Then 
\[\pi_T(i(\sigma)i(b)^li(b)^{*m}i(\tau)^*)e_{\mu b^n} = \begin{cases}e_{\sigma b^{l+(n-m)}} &\text{if $\mu=\tau$ and } n\geq m\\0&\text{otherwise}
                                                       \end{cases}\]
is again in $H_k$. 

For item (\ref{it2:applem3}), take $B=\sum_{\rho, \mu\in\Sigma_k}i(\rho)D_{\rho, \mu}i(\mu)^* \in B_k$ and  suppose $\pi_T(B)|_{H_k} = 0$. Fix $\sigma,\tau\in \Sigma_k$. Then
\[T_{\sigma}^*\pi_T(B) T_{\tau} = \pi_T(i(\sigma)^*)\pi_T(B)\pi_T(i(\tau))= \pi_T(D_{\sigma,\tau}).\]
Since $T_{\tau}$ is an injection from $H_0$ into $H_k$ and $\pi_T(B)|_{H_k} = 0$, we have $\pi_T(B)T_\tau|_{H_0}=0$. 
Thus $\pi_T(D_{\sigma,\tau})|_{H_0} = 0$.  But  the restriction $(\pi_{T}|_{C^*(i(b))})|_{H_0}$ is generated by a nonunitary 
isometry and hence is faithful by Coburn's theorem (see, for example, \cite[Theorem~3.5.18]{Murphy}).  Thus $D_{\sigma,\tau} = 0$. It follows that $B=0$. \end{proof}

\begin{lemma}
 \label{lem:ap4}
$\pi_T$ is faithful on the core $C^*(G,P)^{\theta}$.
\end{lemma}

\begin{proof}
 Since  $C^*(G,P)^{\theta} = \overline{\bigcup_k C_k}$ by Lemma~\ref{cor:ap2}, it suffices to show that
$\pi_T$ is faithful on each $C_k$.  Suppose  $\pi_T(R)=0$ where $R \in C_k$.  Then there exist $R_i\in B_i$ such that $R = R_0 + \dots + R_k$, and then
$\pi_T(R_0) + \dots + \pi_T(R_k) = 0$.

For stems $\sigma$ and $\tau$, if $\theta(\sigma) <\theta(\tau)$ then $T_{\tau}^* e_{\sigma b^n} = 0$. It follows that  $\pi_T(R_i)|_{H_j}=0$
when $j <i$.
Thus 
\[0=\pi_T(R_0)|_{H_0} + \dots + \pi_T(R_k)|_{H_0}=\pi_T(R_0)|_{H_0}, \]
and Lemma~\ref{lem:ap3} implies that $R_0=0$. Then an induction argument gives each $R_i = 0$.  Thus $R=0$, as required.   
\end{proof}

\begin{proof}[Proof of Theorem~\ref{BSamenable}]
We use an argument similar to that of \cite[Proposition~4.2]{LRold}.
Let  $\Phi^{\theta}$ be the conditional expectation of $C^*(G,P)$ onto $C^*(G,P)^{\theta}$ obtained by averaging over the action $\hat\theta$, and recall that $\Phi^{\theta}$ is faithful.
Let $\{E_z\}$ be the usual orthonormal basis for $\ell^{\infty}(P)$.  The diagonal map $\Delta:B(\ell^2(P)) \to \ell^{\infty}(P)$ 
given by
\[\Delta(T) = \sum_{z \in P} E_z TE_z\]  is  faithful, and  $\Delta \circ \pi_T = \pi_T \circ \Phi$ (see the computation on Page~433 of \cite{LRold}).
Now 
\begin{align*}
 \Phi(R^*R) = 0 &\implies \Phi(\Phi^{\theta}(R^*R))= 0 \\&\implies \pi_T(\Phi(\Phi^{\theta}(R^*R)))=0\\& \implies 
 \Delta \circ \pi_T(\Phi^{\theta}(R^*R)) = 0\\& \implies \pi_T(\Phi^{\theta}(R^*R))=0  \text{ because $\Delta$ is faithful}\\
& \implies \Phi^{\theta}(R^*R)=0 \text{ by Lemma~\ref{lem:ap4}}\\
&\implies R=0 \text{ because $\Phi^{\theta}$ is faithful.}
\end{align*}
Thus $\Phi$ is faithful and $(G,P)$ is amenable.
\end{proof}

\begin{cor}\label{same}  The Toeplitz representation $\pi_T:C^*(G,P)\to\mathcal{T}(G, P)$ is faithful.
\end{cor}

\begin{proof} Since $(G, P)$ is amenable by Theorem~\ref{BSamenable}, $\pi_T$ is faithful by \cite[Corollary~3.8]{LRold}.
\end{proof}


\begin{thebibliography}{99}
\bibitem{BR} O. Bratteli and D.W. Robinson, \emph{Operator Algebras and Quantum Statistical Mechanics II}, second ed., Springer-Verlag, Berlin, 1997.

\bibitem{BaHLR} N. Brownlowe, A. an Huef, M. Laca and I. Raeburn, Boundary quotients of the Toeplitz algebra of the affine semigroup over the natural numbers, \emph{Ergodic Theory Dynam. Systems} \textbf{32} (2012), 35--62. 

\bibitem{CM} A. Connes and M. Marcolli, \emph{Noncommutative Geometry, Quantum Fields, and Motives}, Colloquium Publications, Vol. 55, American Mathematical Society, 2008.

\bibitem{CL1} J. Crisp and M. Laca, On the Toeplitz algebras of right-angled  and finite-type Artin groups, \emph{J. Austral. Math. Soc.}  \textbf{72} (2002), 223--245. 

 \bibitem{CL2} J. Crisp and M. Laca, Boundary quotients and ideals of Toeplitz $C^*$-algebras of Artin groups,  \emph{J. Funct. Anal.} \textbf{242} (2007), 127--156.
 
\bibitem{EaHR} R. Exel, A. an Huef and I. Raeburn, Purely  infinite simple $C^*$-algebras associated to integer dilation matrices, \emph{Indiana Univ. Math. J.} \textbf{60} (2011), 1033--1058.
  
\bibitem{aHLRS} A. an Huef, M. Laca, I. Raeburn and A. Sims, KMS states on the $C^*$-algebras of finite graphs, \emph{J. Math. Anal. Appl.} \textbf{405} (2013), 388--399.

\bibitem{KT} E. Kaniuth and K.F. Taylor, \emph{Induced Representations of Locally Compact Groups}, Cambridge Tracts in Mathematics, vol. 197, Cambridge Univ. Press, 2013.

\bibitem{KP}A. Kumjian and D. Pask, Higher rank graph $C^*$-algebras, \emph{New York J. Math.} \textbf{6} (2000), 1--20.

\bibitem{LRold} M. Laca and I. Raeburn, Semigroup crossed products and the Toeplitz algebras of nonabelian groups,  \emph{J. Funct. Anal.} \textbf{139} (1996), 415--440.

\bibitem{LR} M. Laca and I. Raeburn, Phase transition on the Toeplitz algebra of the affine semigroup over the natural numbers, \emph{Adv. Math.} \textbf{225} (2010), 643--688.

\bibitem{LRR} M. Laca, I. Raeburn and J. Ramagge, Phase transition on Exel crossed products associated to dilation matrices, \emph{J. Funct. Anal.} \textbf{261} (2011), 3633--3664.

\bibitem{LS} R.C. Lyndon and P.E. Schupp, \emph{Combinatorial Group Theory}, Springer-Verlag, Berlin, 1977.

\bibitem{Murphy} G.J. Murphy, \emph{{$C^*$}-Algebras and Operator Theory}, Academic Press, Boston, 1990. 

\bibitem{N} A. Nica, $C^*$-algebras generated by isometries and Wiener-Hopf operators, \emph{J. Operator Theory} \textbf{27} (1992), 17--52.

\bibitem{P} G.K. Pedersen, \emph{$C^*$-Algebras and their Automorphism Groups}, London Math. Soc. Monographs, vol. 14, Academic Press, London, 1979.

\bibitem{RSY2} I. Raeburn, A. Sims and T. Yeend, The $C^*$-algebras of finitely aligned higher-rank graphs, \emph{J. Funct. Anal.} \textbf{213} (2004), 206--240.


\bibitem{tfb} I. Raeburn  and D.P. Williams, \emph{Morita Equivalence and Continuous-Trace {$C^*$}-Algebras}, Math. Surveys and Monographs, vol. 60, Amer. Math. Soc., Providence, 1998.


\bibitem{SBS}  J. Spielberg, $C^*$-algebras for categories of paths associated to the Baumslag-Solitar groups, \emph{J. London Math. Soc.} \textbf{86} (2012), 728--754.

\bibitem{SCP} J. Spielberg, Groupoids and $C^*$-algebras for categories of paths, \emph{Trans. Amer. Math. Soc.} \textbf{366} (2014), 5771--5819.

\end{thebibliography}
\end{document}